\documentclass[11pt,a4paper,reqno]{article}
\usepackage[margin=1.8cm]{geometry}
\usepackage{amsfonts,amsthm,amsmath,amscd,amssymb}
\usepackage[mathscr]{eucal}
\usepackage{indentfirst,enumitem,graphicx,epstopdf,subcaption,hyperref}
\usepackage{pifont, dsfont,bm,mathtools,xcolor,empheq,float,hyperref}
\usepackage[title]{appendix}
\numberwithin{equation}{section}
\theoremstyle{plain}

\newtheorem{Th}{Theorem}[section]
\newtheorem{Lemma}[Th]{Lemma}
\newtheorem{Cor}[Th]{Corollary}

\newtheorem{assumption}[Th]{Assumption}

\graphicspath{{plot/}}
\newcommand{\bS}{\mathbb{S}}

\newcommand{\bod}{\mathbf{d}}
\newcommand{\ty}{\tilde{y}}

\title{A Weighted Sampling Method for Inverse Medium Problem \\ with Limited Aperture}

\author{ Fuqun Han \footnote{Department of Mathematics, University of California Los Angeles, Los Angeles, California.  ({fqhan@math.ucla.edu}).}
\and Kazufumi Ito \footnote{Department of Mathematics,
North Carolina State University, Raleigh, North Carolina.  ({kito@ncsu.edu}).}}

\begin{document}
\date{}
\maketitle
\abstract{
Inverse medium scattering problems arise in many applications, but in practice, the measurement data are often restricted to a limited aperture by physical or experimental constraints. 
Classical sampling methods, such as MUSIC and the linear sampling method, are well understood for full-aperture data, yet their performance deteriorates severely under limited-aperture conditions, especially in the presence of noise. 
We propose a new sampling method tailored to the inverse medium problem with limited-aperture data. The method is motivated by the linear sampling framework and incorporates a weight function into the index function. The weight is designed so that the modified kernel reproduces the full-aperture behavior using only limited data, which both localizes oscillations and improves the conditioning of the far-field system, thereby yielding more accurate and stable reconstructions. We provide a theoretical justification of the method under the Born approximation and an efficient algorithm for computing the weight. Numerical experiments in two and three dimensions demonstrate that the proposed method achieves greater accuracy and robustness than existing sampling-type methods, particularly for noisy, limited-aperture data.

\medskip
\noindent\textbf{Keywords:} Inverse medium scattering, limited-aperture data, sampling method, weighted kernel.
}

\section{Introduction}

Inverse acoustic medium scattering arises in a wide range of applications, including radar and sonar imaging \cite{sonar}, microwave tomography for tissue characterization \cite{microwave_tomo}, geophysical exploration \cite{geophysical_inverse}, nano-optics \cite{novotny2012principles}, and nondestructive testing \cite{non_destructive_test_book}. 
In many of these settings, measurement data are available only over a limited range of incident and observation angles. 
For example, in sonar imaging and oil exploration, transducers are typically located only above the region of interest, and the effective aperture decreases as the target lies deeper underwater or underground. 
In tomographic phase microscopy, the illumination angle is limited to about $60^\circ$ due to the numerical aperture of the condenser lens \cite{TPM_review}. 
Consequently, developing reconstruction methods that remain accurate and stable with limited-aperture measurements is crucial

Traditional approaches formulate the inverse medium problem as the minimization of a data-fidelity functional with regularization. 
While effective in certain regimes, such optimization-based methods are computationally expensive and prone to local minima. 
To overcome these difficulties, a variety of non-iterative sampling methods have been developed, such as MUSIC \cite{music_first_2003}, the linear sampling method (LSM) \cite{first_LSM_1996} and its generalizations \cite{GLSM_first,garnier2023linear}, the factorization method \cite{factorization_book}, and reverse time migration \cite{RTM_EM_2013}; see also the surveys \cite{book_cakoni, xudong_book_2018}. 
These methods construct an index function that attains large values inside the inhomogeneous medium and small values outside. 
Although several extensions have been proposed for limited-aperture data, their performance typically degrades significantly compared with the full-aperture setting
\cite{music_limited_Park, GLSM_limited_aperture, dsm_limited}. 
Alternative strategies include data-completion and view-reconstruction techniques \cite{dou2022data} and physics-aware deep-learning methods \cite{yin2025physics}, though these often incur higher computational costs. 
This motivates the development of sampling methods specifically adapted to limited-aperture settings.

In this work, we propose a weighted linear sampling method for reconstructing the support of an inhomogeneous medium using data collected over a limited range of angles. 
Our approach is inspired by the classical LSM \cite{book_cakoni} in the full-aperture case and by direct sampling methods \cite{ito2012direct,EIT_DSM,decoupling_dsm}, which exploit suitably designed inner products to extract reconstruction information from limited data.

For the classical LSM, let $F$ denote the far-field operator associated with the far-field pattern $u^{\infty}$. 
If $g_x$ satisfies $F g_x = \Psi_x$, where $\Psi_x$ denotes the far-field pattern of the fundamental solution, the LSM defines the index function
\begin{equation}
 I_{LSM}(x) := \frac{1}{\| g_x\|^{2}_{L^2(\mathbb{S}^{d-1})}}\,,
\end{equation}
which indicates the support of the medium. 
A major difficulty in the limited-aperture case is that the discretized far-field operator (or equivalently, the data matrix) becomes much more ill-conditioned. 
Although some methods have been tested with noiseless limited-aperture data, they tend to be highly sensitive to measurement errors.

To mitigate the ill-posedness inherent in limited-aperture data, we introduce a weighted kernel that plays the role of a surrogate for the full-aperture kernel underlying the classical LSM. 
Recall that in the full-aperture case, the kernel $\langle \Psi_x, \Psi_z \rangle_{L^2(\mathbb{S}^{d-1})}$ is strongly peaked when $x=z$ and small otherwise, so that the corresponding index function provides sharp localization of the medium. 
When the measurement aperture is limited, this localization deteriorates. 
To restore it, we modify the inner product by inserting a weight function $w$, and define for $x,z \in \Omega$
\[
K_w(x,z) := \langle \Psi_x,\, w \Psi_z \rangle_{L^2(\Gamma)},
\]
where $\Gamma$ is the measurement surface and $\Psi$ is the fundamental solution of the forward problem. 
The weight $w$ is chosen so that $K_w(x,z)$ recovers the desired localization property, namely that it attains a strong peak at $x=z$ and is small otherwise, thereby mimicking the behavior of the full-aperture kernel. 
This construction improves the conditioning of the associated far-field equation and leads to an index function that is significantly more stable and accurate under limited-aperture measurements.

The remainder of the paper is organized as follows. 
Section~\ref{sec_formulation} formulates the inverse problem, motivates the weighted kernel construction, and introduces the index function. 
Section~\ref{sec_accuracy} analyzes the accuracy of the index function, beginning with the continuous full-data setting and then addressing the discrete finite-data case under the Born approximation. 
Section~\ref{sec_construction_w} presents an efficient computational procedure for constructing the weight function and examines its behavior under noisy data. 
Finally, Section~\ref{sec_numerics} reports two- and three-dimensional numerical experiments that validate the proposed method and confirm the theoretical predictions under noisy and limited-aperture conditions.

\section{Inverse medium scattering problem with limited aperture}
\label{sec_formulation}
In this work, we focus on the scalar wave equation (with $d=2,3$)
\begin{equation}
\label{model_equation}
  \Delta u + k^2(1+q)u = 0 \quad \text{in} \quad \mathbb{R}^d\,,
\end{equation}
to develop a new weighted sampling method, where $k$ is the wave number and $q\in L^2(D)$ represents the inhomogeneous medium that will be recovered.
Consider the incident field $u^{i}(y,\bod) = e^{ik y \cdot \bod}$ for $\bod \in \bS^{d-1}$ as the incidence direction, then the scattered field is denoted as $u^s = u- u^{i}$.

The far-field pattern $u^\infty$ describes the asymptotic behavior of the scattered wave $u^s$ when the observation point is far away from the inclusion. Specifically, $u^\infty$ is related to $u^s$ by
\begin{equation}
    u^s(y,\bod) = \frac{e^{ik|y|}}{|y|^{(d-1)/2}}\left(u^{\infty}(\hat{y},\bod) + \mathcal{O}\!\left(\frac{1}{|y|}\right)\right),
    \qquad \hat{y} = \frac{y}{|y|}\,.
\end{equation}

We shall also often need the fundamental solution $\Phi(z,y)$ and its far field pattern associated with the operator $-\Delta-k^2$ under the radiation condition
\begin{equation}
\label{def_Phi}
    \Phi(z,y) = \begin{cases}
    e^{ik|z-y|}/(4\pi |z-y|)\,  &\text{ for } d=3\,; \\
      \frac{i}{4}H_0^1(k|z-y|)\,  &\text{ for } d=2\, .
    \end{cases}\qquad 
    \Phi^{\infty}(z,\hat{y}) = \begin{cases} \frac{1}{4\pi}  e^{-ik z\cdot\hat{y}}
    \,  &\text{ for } d=3\,; \\
    \frac{e^{i\pi/4}}{\sqrt{8\pi k}}   e^{-ik z\cdot \hat{y}}\,  &\text{ for } d=2\, .
    \end{cases}
\end{equation}
In the following, we write $\Psi_z(y) = \Phi^{\infty}(z,\hat{y})$ to simplify the notation.
Using the fundamental solution $\Phi(z,\cdot)$ to \eqref{model_equation}, we can represent the total field $u$ as
\begin{equation}
\label{int_equ_u}
    u(y,\bod) = k^2\int_{D}q(z)\Phi(z,y)u(z,\bod)\,dz + u^i(y,\bod)\,,
\end{equation}
which yields the alternative expression for the far-field pattern
\begin{equation}
\label{far_field_integral}
u^{\infty}(\hat{y},\bod)  = k^2\gamma\int_{D}q(z)\, u(z,\bod)\,e^{-ik z \cdot \hat{y}}\, dz\,,
\quad 
\text{for } \hat{y}, \,\bod \in \bS^{d-1}\,,
\end{equation}
where $\gamma$ is a dimension-dependent constant.

 The inverse medium scattering problem of our interest is to recover the support of $q$ from the far-field pattern $u^{\infty}(\hat{y},\bod)$, which is assumed to be available over a surface $\Gamma\subset \bS^{d-1}$, i.e., only the measurement data with a limited aperture is available. 
In particular, we introduce the measurement aperture
\begin{equation}
\label{def_Gamma}
\Gamma :=\begin{cases}
\{ (\cos\theta,\,\sin\theta):\,  \theta \in [-\alpha,\alpha]\} \subset \mathbb{S}^1 &  \text{in } \mathbb{R}^2\,, \\
\{ (\cos\phi\sin\theta,\, \sin\phi\sin\theta,\, \cos\theta):\, \phi \in [-\alpha,\alpha],\, \theta \in [0,\beta]\} \subset \mathbb{S}^{2}  & \text{in } \mathbb{R}^3\,,
\end{cases}
\end{equation}
where the parameters $\alpha$ (and $\beta$ in three dimensions) specify the measurement aperture.

The far-field operator is now defined by
\begin{equation}
\label{def_far}
(Fv)(y) := \int_{\Gamma} u^{\infty}(y,\bod) v(\bod) d\bod\,,
\end{equation}
for $v \in L^2(\Gamma)$ and $y \in \Gamma$ when restricting the integration on $\Gamma$.
For brevity, we omit the hat notation on $y$.

We emphasize that classical sampling and factorization methods—such as the linear sampling method, MUSIC, and the factorization method—are all based on suitable factorizations of this far-field operator together with range conditions \cite{book_cakoni,factorization_book,music_first_2003}.
 
For convenience, we first recall the classical linear sampling method \cite{book_cakoni}:
\begin{itemize}
    \item Construct a set of sampling points $\Omega_{N_x}:=\{x_j\}_{j=1}^{N_x}$ covering the domain of interest $\Omega$.
    \item For each $x \in \Omega_{N_x}$, solve $F g_x = \Psi_x$, where $\Psi_x$ is often referred to as a test function. 
    \item Define the index function
    \begin{equation}
    \label{def_index_LSM}
    I_{LSM}(x) := \frac{1}{\| g_x \|^{2}_{L^2(\Gamma)}} \,,
    \end{equation}
    which serves as an indicator: it attains relatively large values when $x$ lies inside the support of $q(x)$, and remains comparatively small when $x$ lies outside.
\end{itemize}

It is well known that the linear sampling method performs reliably when full-aperture data are available and the measurement noise is relatively small. In contrast, reconstructions from limited-aperture data are often unsatisfactory and highly sensitive to random noise; this issue will also be illustrated in the numerical experiments of Section~\ref{sec_numerics}. The underlying reason is that the far-field operator $F$ in \eqref{def_far} becomes increasingly ill-posed as the aperture decreases.

This observation motivates the present work. We introduce a new index function based on a modified family of testing functions $\eta_x$ and a modified far-field operator $F_w$. The operator $F_w$ is designed so that its discretization exhibits better conditioning in terms of its singular values. As a result, the norm of the solution $g^w_x$ to the modified far-field equation
$
F_w g^w_x = \eta_x
$
leads to a more effective index function in the setting of noisy, limited-aperture measurements. This modification aims to enhance both the accuracy and the stability of the reconstruction.

To better illustrate the motivation outlined above, we now consider a simplified setting under the Born approximation for the scattered field, i.e., $u \approx u^i$, and further assume $q(x) \equiv 1$ in $D$.  
In this case, the far-field pattern in \eqref{far_field_integral} for a measurement angle $y\in \mathbb{S}^{d-1}$ reduces to  
\begin{equation}
\label{far_field_approx}
    u^{\infty}(y,\bod) \;\approx\; k^2\gamma \int_{D} e^{ik(\bod-y)\cdot z}\, dz \, .
\end{equation}

Within this approximation, one obtains the factorization
$
F \;=\; k^2\gamma \, H H^* \quad \cite{kirsch2008factorization},
$
where the operator $H : L^2(\Gamma)\to L^2(D)$ is defined by  
\begin{equation}
\label{def_H}
    (H\phi)(z) := \int_{\Gamma} \phi(y)\, e^{ik z\cdot y}\, d\sigma_y,
    \qquad z\in D\,,
\end{equation}
and $H^*$ denotes the adjoint of $H$.  

Introducing the kernel function $K$ by  
\begin{equation}
\label{def_K}
    K(x,z) := (H^*H)(x,z) 
    \;=\; \int_{\Gamma} e^{ik(x-z)\cdot y}\, d\sigma_y
    \;=\; \langle \Psi_x, \Psi_z \rangle_{L^2(\Gamma)}\, ,
\end{equation}
we observe that $(k^2\gamma)K$ and the far-field operator $F$ share the same set of singular values after discretization.  
This follows directly from the singular value decomposition of $H$, since $F = k^2\gamma HH^*$ and $K = H^*H$.  

To analyze $K(x,z)$, we recall the Jacobi-Anger identity: 
\begin{equation}
\label{Jacobi_anger}
     e^{ikx\cdot y} \;=\;
     \begin{cases}
     \displaystyle \sum_{n\in \mathbb{Z}} i^n J_n(k|x|)\, e^{in(\theta_x-\theta_y)}, 
     & \text{in } \mathbb{R}^2\,, \\[2ex]
     \displaystyle 4\pi \sum_{n\geq 0} i^n j_n(k|x|) 
     \sum_{m=-n}^n Y_n^m(\hat{x}) \, \overline{Y_n^m(y)},
     & \text{in } \mathbb{R}^3\,,
     \end{cases}
\end{equation}
where $x\in \mathbb{R}^d$ and $y\in \mathbb{S}^{d-1}$.  

Specializing to $\mathbb{R}^2$ and a fixed aperture $\Gamma$, to understand the impact of aperture truncation, we expand the kernel $K(x,z)$ in Fourier modes as 
\begin{equation}
\label{similarlity_gamma}
  K(x,z) = \langle \Psi_x, \Psi_z \rangle_{L^2(\Gamma)} 
   = 2\alpha J_0(k|x-z|) \;+\; \sum_{n\geq 1} i^n J_n(k|x-z|)\, e^{in(\theta_x-\theta_z)}\, \frac{2\sin(n\alpha)}{n}\,.
\end{equation}
Thus $K(x,z)$ consists of the leading term $J_0(k|x-z|)$, corresponding to the full-aperture case, plus a contribution from higher-order Bessel functions arising from aperture truncation by $\alpha < \pi$.

From the definitions of $F$ and $K$, it follows that the accuracy and stability of solving $F g_x = \Psi_x$ is governed by the degree of ill-posedness of $K$, as given in \eqref{def_K}–\eqref{similarlity_gamma}. In the full-aperture case $\Gamma = \bS$ with $x,z\in \Omega$, the series in \eqref{similarlity_gamma} collapses and we obtain
\[
  K(x,z) = 2\pi J_0(k|x-z|)\,,
\]
which attains its unique maximum at $x=z$ and decays monotonically as $|x-z|$ increases. Consequently, if the sampling points in $\Omega$ are well-separated, $K$ is diagonally dominant, and the discrete far-field operator $F$ has a relatively small condition number. Hence, the computation of the index function via $F g_x = \Psi_x$ is numerically stable.

By contrast, for a limited aperture $\Gamma$, the tail terms in \eqref{similarlity_gamma} no longer vanish. In this case, $K(x,z)$ may fail to attain its maximum at $x=z$, and its decay away from $x=z$ is significantly slower. In the neighborhood of the origin with $r>0$, the Bessel function $J_0(r)$ decreases as $r$ increases, whereas the higher-order modes $J_n(r)$ for $n>0$ grow with $r$, thereby diminishing the localization induced by $J_0$.
As a result, $F$ becomes highly ill-conditioned under limited-aperture measurements, rendering the numerical solution of the far-field equation unstable, even with regularization.

\medskip 

To address this difficulty, we introduce a weight function $w \in L^{\infty}(\Gamma)$, to be determined, and consider the weighted kernel
\begin{equation}
\label{def_Kw}
    K_w(x,z) := \langle \Psi_x,\, w \Psi_z \rangle_{L^2(\Gamma)}  
    = \int_{\Gamma} e^{ik(x-z)\cdot y}\, w(y)\, d\sigma_y\,.
\end{equation}
As we shall show, with a suitably chosen weight $w$ satisfying
\begin{equation}
\label{K_W_approx}
    K_w(x,z) \approx J_0(k|x-z|)\,,
\end{equation}
the kernel $K_w$ is sharper than the unweighted kernel $K$ in \eqref{def_K}, and its singular values decay more slowly. Moreover, the construction of such a weight $w$ is relatively straightforward, as will be demonstrated in Section~\ref{sec_construction_w}. We emphasize that the requirement \eqref{K_W_approx} is to recover the full aperture kernel and enables the analytical construction of candidate weights $w$, which in turn allows us to study the theoretical properties of the weighted kernel (see Section~\ref{sec_finite}).

Correspondingly, we define the weighted far-field operator by
\[
    F_w := w^{1/2} F w^{1/2}\,,
\]
and observe that $F_w$ and $K_w$ share the same singular values after discretization under the Born approximation. With an appropriate weight $w$, the modified far-field equation
\[
    F_w g^w_x = \eta_x
\]
can be solved more stably and lead to a more accurate reconstruction.  

Before presenting the theoretical justification, we summarize the proposed algorithm.

\medskip 
\noindent \textbf{Weighted Linear Sampling Method.}
\begin{itemize}
\item Let $\Gamma_{N_y} = \{y_j\}_{j=1}^{N_y}$ be a discrete set of measurement points on $\Gamma$. The measurement data consists of $u^{\infty}(y,\bod)$ for $y,\bod \in \Gamma_{N_y}$.  

\item Let $\Omega_{N_x} = \{x_j\}_{j=1}^{N_x}$ be a set of sampling points in $\Omega$. For each $x \in \Omega_{N_x}$, solve
\begin{equation}
\label{def_gzw}
    F_w g^w_x = \eta_x\,, 
    \qquad F_w := w^{1/2} F w^{1/2}\,, \qquad \eta_x := w^{1/2}\Psi_x\,,
\end{equation}
using Tikhonov regularization, where $w$ satisfies \eqref{K_W_approx} (see Section \ref{sec_construction_w} for details).

\item For each $x \in \Omega_{N_x}$, compute the index function
\begin{equation}
\label{def_index}
    I_w(x) := \frac{1}{\| g^w_x \|^2_{L^2(\Gamma)}}\,,
\end{equation}
which serves as an indicator to reconstruct the support of the inhomogeneous medium.
\end{itemize}

\section{Accuracy analysis of the weighted index function}
\label{sec_accuracy}

In this section, we examine the accuracy of the index function defined in \eqref{def_index} in the setting of limited-aperture data.
In Section~\ref{sec_infinite}, we begin with an idealized scenario in which sufficiently many measurement data within the limited aperture are available so that the far-field operator can be characterized analytically. In this case, we show that the proposed index function correctly identifies the inclusions through certain range conditions of the far-field operator.
In Section~\ref{sec_finite}, we then address the practically relevant situation where both the boundary measurement points and the incident waves are finite in number. For this setting, under Born approximation, we explicitly construct the weight function and derive rigorous estimates for the kernel $K_w$ and the index function, thereby confirming its accuracy for representative classes of inhomogeneous medium.

\subsection{Continuous analysis with infinite measurements} 
\label{sec_infinite}

We first analyze the case where the far-field operator $F$ defined in \eqref{def_far} is known on the measurement surface $\Gamma$ in \eqref{def_Gamma}.  
Let $H_w := H w^{1/2}$, with $H$ given in \eqref{def_H}.  
The following theorem characterizes the behavior of solutions to the weighted far-field equation $F_w g_x^w = \eta_x$, where $F_w$ and $\eta_x$ are defined in \eqref{def_gzw}.

\begin{Th}
\label{Thm_LSM}
Assume that $k$ is not a Dirichlet eigenvalue of $-\Delta$ in $D$, and that the weight function 
$w \in L^{\infty}(\Gamma)$ satisfies $\text{essinf}_{y \in \Gamma} |w(y)| > 0$.  
Let $F_w$ and $\eta_x$ be defined as in \eqref{def_gzw}. Then:
\begin{itemize}
    \item If $x \in D$, there exists a sequence $g_{x,\tau}^w \in L^2(\Gamma)$ such that
    \begin{equation}
       \lim_{\tau \to 0} \| F_w g_{x,\tau}^w - \eta_x \|_{L^2(\Gamma)} = 0\,,
    \end{equation}
    and
    \begin{equation}
    \label{finite_Hg}
        \lim_{\tau \to 0} \| H_w g_{x,\tau}^w \|_{L^2(D)} < \infty \,.
    \end{equation}
    \item If $x \in \mathbb{R}^d \setminus D$, then for every sequence $g_{x,\tau}^w \in L^2(\Gamma)$ satisfying
    \begin{equation}
        \lim_{\tau \to 0} \| F_w g_{x,\tau}^w - \eta_x \|_{L^2(\Gamma)} = 0\,,
    \end{equation}
   one has  
    \begin{equation}
    \label{blow_up_Hg}
        \lim_{\tau \to 0} \| H_w g_{x,\tau}^w \|_{L^2(D)} = \infty \,.
    \end{equation}
\end{itemize}
\end{Th}

\begin{proof}
In the full-aperture case, it is shown in \cite{book_cakoni} that the factorization $F = G H$ holds, where $G : \overline{R(H)} \to L^2(\Gamma)$ and $\overline{R(H)}$ is the closure of the range of $H$ in \eqref{def_H}.  
Moreover, $G$ satisfies $G(u_0) = p^{\infty}$, where $u_0$ and $p$ solve
\begin{equation}
   \begin{cases} 
   \Delta p + k^2(1+q)p = -k^2 q u_0 \quad &\text{in } \mathbb{R}^d, \\
   \lim_{R \to \infty}\int_{|x|=R}\Bigl|\tfrac{\partial p}{\partial |x|} - i k p\Bigr|^2 ds = 0 \,.
   \end{cases}
\end{equation}
From this, one deduces that $x \in D$ if and only if $\Psi_x \in R(G)$; equivalently, there exists $u_x$ such that $G u_x = \Psi_x$.  
Since $H$ is compact and injective, with range 
\[
\overline{R(H)} = \{ v \in L^2(D) : \Delta v + k^2 v = 0 \text{ in } D\}\,,
\]
one can find $g_{x,\tau}$ such that $H g_{x,\tau} \to u_x$ as $\tau \to 0$.  

For the weighted case, let $M_{w^{1/2}}:L^2(\Gamma)\to L^2(\Gamma)$ denote multiplication by $w^{1/2}$.
By the hypothesis, the operator $M_{w^{1/2}}$ is a bounded bijection with bounded inverse $M_{w^{-1/2}}$ (choose any measurable branch of $w^{1/2}$).
Define
\[
H_w := H\,M_{w^{1/2}}\,,\qquad G_w := M_{w^{1/2}}\,G\,,
\]
so that $F_w = G_w H_w$. Because $M_{w^{1/2}}$ is bounded and invertible, composition with it
preserves compactness and injectivity: $H_w$ is compact and injective and $R(H_w)=R(H)$.
Similarly, $R(G_w)=M_{w^{1/2}}(R(G))$, and hence $\eta_x=w^{1/2}\Psi_x\in R(G_w)$ if and only if
$\Psi_x\in R(G)$. The same range condition above, therefore, applies to the weighted case, yielding the two assertions of the theorem.
\end{proof}

Thus, the weighted formulation retains the same range characterization as the classical LSM, ensuring that the index function distinguishes points inside and outside the inclusion. We remark that since we only assume $\text{essinf}_{y\in\Gamma}|w(y)|>0$, the weight $w$ may be sign-changing or complex-valued. 
Any measurable choice of the square root $w^{1/2}$ is valid, and different choices differ only by a factor of unit modulus, which has no effect on the analysis.

In practice, the sequence $g_{x,\tau}^w$ is obtained by solving the Tikhonov-regularized equation
\[
(F_w^* F_w + \tau I) g_{x,\tau}^w = F_w^* \Psi_x,
\]
and the index function is defined as $\| g_{x,\tau}^w \|_{L^2(\Gamma)}^{-2}$ for the reconstruction of the inhomogeneous medium.  
Since $H$ is compact, the blow-up property \eqref{blow_up_Hg} also manifests as the unbounded growth of $\| g_{x,\tau}^w \|_{L^2(\Gamma)}$ when $\tau \rightarrow 0$.

A limitation of Theorem~\ref{Thm_LSM} is that it assumes complete knowledge of the far-field operator $F$, which is not available in practice.  
When only finitely many measurements are collected, the discrete far-field operator is represented by a finite-dimensional matrix, which typically has full rank. In this case, the blow-up property \eqref{blow_up_Hg} no longer holds.  
For this reason, the next section addresses the practically relevant scenario where the far-field operator is discretized from a finite number of measurements.  
 
\subsection{Discrete analysis with finite measurements}
\label{sec_finite}

In this subsection, we analyze the accuracy of the proposed index function in the practically relevant setting where both the number of measurement points on the boundary and the number of incidence directions are finite. For a finite-dimensional vector $u$, we also denote by $\|u\|_2$ its discrete $\ell^2$-norm.   
In the following, we first define the discrete setting and justify assumptions in Section \ref{sec_discrete_assump}. Then we construct weights ensuring kernel localization in Section \ref{sec_weights_constr}. Next, we analyze approximate solutions and error bounds in Section \ref{sec_error_sol}, and finally extend to more general inclusions in Section \ref{sec_gen_analysis}.

\subsubsection{Discrete setting and assumptions}
\label{sec_discrete_assump}
We begin with an estimate for the scattered field to justify the Born approximation. 
By \cite{bao2000regularity}, the scattered wave solving \eqref{model_equation} satisfies
\[
\|u^s\|_{L^2(D)} \;\le\; C_1\,|D|^{1/2}\,\|q\|_{L^\infty(D)}\,\|u^i\|_{H^2(D)}\,,
\]
where $C_1$ depends only on the wavenumber $k$. Using the identity
\[
u^\infty(y,\mathbf d)-k^2\gamma\int_D q(z)u^i(z,\mathbf d)e^{-iky\cdot z}\,dz
= k^2\gamma\int_D q(z)u^s(z,\mathbf d)e^{-iky\cdot z}\,dz\,,
\]
and applying H\"older's and Cauchy--Schwarz inequalities, we obtain
\begin{equation}
\label{approx_far_field}
\bigg|u^\infty(y,\mathbf d)-k^2\gamma\int_D q(z)u^i(z,\mathbf d)e^{-iky\cdot z}\,dz\bigg|
\;\le\; C\,|D|\,\|q\|_{L^\infty(D)}^{2}\,\|u^i\|_{H^2(D)} ,
\end{equation}
with $C = k^2|\gamma| C_1$. 
Estimate \eqref{approx_far_field} shows that when either $|D|$ or $\|q\|_{L^\infty(D)}$ is small, 
the contribution of $u^s$ is negligible and the far-field pattern $u^\infty$ depends approximately linearly on $u^i$.

\begin{assumption}[Born approximation]
\label{assump:born}
We assume that the error bound in \eqref{approx_far_field} is negligible, 
so that the far-field pattern may be approximated by
\begin{equation}
\label{eqn_Born}
u^\infty(y,\mathbf d)\;\approx\; k^2\gamma\int_D q(z)\, e^{ik(z\cdot \mathbf d - y\cdot z)}\,dz ,
\end{equation}
that is, $u^\infty$ depends linearly on the incident wave $u^i$.
\end{assumption}

To derive a discrete formulation, let $D_{N_z} := \{z_n\}_{n=1}^{N_z}$ be a collection of points in $D$, 
and $\Gamma_{N_y} := \{y_j\}_{j=1}^{N_y}$ denote the set of boundary measurement points, and assume $q(x)$ is constant in $D$. 
Supposing $N_z > N_y$, the far-field operator $F$ defined in \eqref{def_far} can be approximated under the discrete Born approximation by
\begin{equation}
\label{def_U}
    F \;\approx\; \tilde{\gamma}\, U^* U, 
    \qquad U_{n,j} = e^{ik z_n \cdot y_j}, 
    \quad \tilde{\gamma} = \frac{k^2 \gamma\, |D|\, q}{N_z},
\end{equation}
where $U \in \mathbb{C}^{N_z \times N_y}$. 
Here the prefactor $\tilde{\gamma}$ arises from replacing the volume integral in \eqref{eqn_Born} by a quadrature rule with uniform weights $|D|/N_z$ over the sampling points $\{z_n\}_{n=1}^{N_z}$ in $D$. 
In particular, $\tilde{\gamma}$ is independent of the measurement aperture.
 If $q(x)$ is not constant, the decomposition becomes $F \approx U^* Q U$ for a diagonal matrix $Q$, and most of the following discussion remains valid.

We now justify the linear independence of the columns of $U$. 
Let $\{\beta_j\}_{j=1}^{N_y} \subset \mathbb{C}$ satisfy
\[
\sum_{j=1}^{N_y} \beta_j \, e^{ik z_n \cdot y_j} = 0 \quad \text{for all } n = 1,\dots,N_z\,.
\]
Define
\[
f(z) := \sum_{j=1}^{N_y} \beta_j \, e^{ik z \cdot y_j}, \qquad z \in \mathbb{R}^d\,.
\]
Since the $y_j$ are distinct directions and the $z_n$ are distinct points in a non-degenerate region of $D$, the exponential $\{e^{ik z \cdot y_j}\}_{j=1}^{N_y}$ are linearly independent on any open set containing the $z_n$. 
Consequently, $f(z_n) = 0$ for all $n$ implies $f \equiv 0$, and therefore $\beta_j = 0$ for all $j$. 
This proves that the columns of $U$ are linearly independent, and hence $\mathrm{rank}(U) = N_y$.

Next, we consider the operator $K_w$ with weight $w$ defined in \eqref{def_Kw}. 
To ensure that all measurement points contribute to the reconstruction, we assume $w(y_j) \neq 0$ for every $y_j \in \Gamma_{N_y}$. 
Our goal is to construct $w$ such that
\begin{equation}
\label{approx_Kw}
    K_w(x,z) \approx J_0(k|x-z|) = \frac{1}{2\pi} \langle \Psi_x, \Psi_z \rangle_{L^2(\mathbb{S}^{d-1})}\,,
\end{equation}
where the right-hand side denotes the $L^2$ inner product of the continuous test functions over the full unit sphere.  

For a quick motivation in the discrete setting, introduce the diagonal weight matrix
\[
   W := \operatorname{diag}\bigl(w(y_1),\dots,w(y_{N_y})\bigr)\in\mathbb C^{N_y\times N_y},
\]
and recall that \(U\in\mathbb C^{N_z\times N_y}\) with entries \(U_{n,j}=e^{ik z_n\cdot y_j}\).
We represent the solution to \eqref{def_gzw} as
\[
   g_x^w \;=\; W^{1/2} U^* \alpha_x, \qquad \alpha_x \in \mathbb{C}^{N_z},
\]
so that \(U^*\alpha_x\in\mathbb C^{N_y}\) and \(g_x^w\in\mathbb C^{N_y}\) as required.

Using the discrete Born approximation \(F\approx\tilde\gamma\,U^*U\) and \(F_w=W^{1/2}FW^{1/2}\),
the weighted equation \(F_w g_x^w=\eta_x\) with \(\eta_x=W^{1/2}\Psi_x\) becomes
\[
   \tilde\gamma\, W^{1/2} U^* U W U^* \alpha_x \;=\; W^{1/2}\Psi_x.
\]
Left-multiplying by \(U W^{1/2}\) (which maps \(\mathbb C^{N_y}\) to \(\mathbb C^{N_z}\)) yields the
\(N_z\times N_z\) linear system
\begin{equation}
\label{KD_partial}
   \tilde\gamma\,(U W U^*)(U W U^*)\,\alpha_x \;=\; U W \Psi_x.
\end{equation}
Solving \eqref{KD_partial} for \(\alpha_x\) determines the testing vector via \(g_x^w=W^{1/2}U^*\alpha_x\).

For the norm of $g_x^w$, we note
\[
   \|g_x^w\|_2^2 = (W^{1/2}U^*\alpha_x)^*(W^{1/2}U^*\alpha_x)
   = \alpha_x^* (U W U^*) \alpha_x .
\]
In the full-aperture setting one obtains the analogous system with \(W\) replaced by the identity
and \(U\) replaced by the corresponding full-aperture sampling matrix \(U_{\mathbb S^{d-1}}\).
Hence, when \(U W U^*\) and \(U_{\mathbb S^{d-1}}U_{\mathbb S^{d-1}}^*\) (and the right-hand sides)
are close, the solutions \(\alpha_x\) and \(\tilde\alpha_x\) are close and consequently
\(\|g_x^w\|_2^2=\alpha_x^*(U W U^*)\alpha_x\) is close to the full-aperture quantity
\(\tilde\alpha_x^*(U_{\mathbb S^{d-1}}U_{\mathbb S^{d-1}}^*)\tilde\alpha_x\).
Therefore, the weighted index inherits the localization properties of the full-aperture index.

\subsubsection{\texorpdfstring{Weight construction and approximation to $J_0$}{Weight construction and approximation properties}}
\label{sec_weights_constr}

To facilitate the analysis, we focus on the two-dimensional setting $\mathbb{R}^2$, while noting that the extension to $\mathbb{R}^3$ requires only minor modifications. 
For convenience, we set $N_y = 2N+1$. 
Our objective is to construct a weight function $w$ such that the corresponding weighted kernel reproduces the localization property of the full-aperture kernel. 
This requires imposing the conditions
\begin{equation}
\label{requirement_w}    
     \sum_{j=1}^{N_y} e^{i n \theta_{y_j}}\, w(y_j) = 
     \begin{cases}
     1, & n = 0, \\[0.3em]
     0, & n = -N, \dots, -1, 1, \dots, N \,,
     \end{cases}
\end{equation}
together with the requirement that the quantities
\[
    \sum_{j=1}^{N_y} e^{i n \theta_{y_j}}\, w(y_j)
\]
remain uniformly bounded for all $|n| > N$. 
Under these conditions, we will show the weighted kernel $K_w(x,z)$ yields the desired approximation stated in \eqref{approx_Kw}.

To see that, using the Jacobi-Anger identity \eqref{Jacobi_anger}, we have
\begin{align}
\label{approx_w_Van}
  \langle \Psi_x, w \Psi_z \rangle_{L^2(\Gamma)} 
   &= \sum_{j=1}^{N_y} e^{ik(x-z) \cdot y_j} w(y_j) \notag \\
   &= \sum_{j=1}^{N_y} \sum_{n \in \mathbb{Z}} i^n J_n(k|x-z|) 
      e^{in(\theta_{x-z} - \theta_{y_j})} w(y_j) \notag \\
   &= J_0(k|x-z|) + \sum_{j=1}^{N_y} \sum_{|n| > N} 
      i^n J_n(k|x-z|) e^{in(\theta_{x-z} - \theta_{y_j})} w(y_j)\,,
\end{align}
where the first equality corresponds to a quadrature rule for the $L^2$ inner product on $\Gamma$. From \eqref{approx_w_Van}, it follows that the weighted $L^2$ inner product provides a good approximation to the inner product for full aperture, provided that the residual term with $|n|>N$ is sufficiently small.

To find $w$ that satisfies \eqref{requirement_w}, note that it is equivalent to a polynomial interpolation problem. Let $\lambda = e^{i d_{\theta}}$ with $d_{\theta} = 2\alpha/(N_y-1)$, and define $y_j = \lambda^{\,j-N-1}$ for $j=1,\dots,N_y$. Then \eqref{requirement_w} reduces to the Vandermonde system with a right-hand side that is $1$ for the $n=0$ mode and $0$ otherwise
\begin{equation}
\label{Vandermonde_matrix}
Vw = \begin{pmatrix}
    1 & 1 & \cdots & 1 \\
    1 & \lambda & \cdots & \lambda^{N_y-1} \\
    1 & \lambda^2 & \cdots & \lambda^{2(N_y-1)} \\
    \vdots & \vdots & \ddots & \vdots \\
    1 & \lambda^{2N} & \cdots & \lambda^{2N(N_y-1)}
\end{pmatrix}
\begin{pmatrix}
   w(y_1)\\ w(y_2)\\ \vdots\\ w(y_{N_y})
\end{pmatrix}  
=       
\begin{pmatrix}
   0\\0\\\vdots \\1\\\vdots\\0
\end{pmatrix}.
\end{equation}
Thus, the weight vector $w$ is determined (up to numerical conditioning) by solving this linear system.

Since $V$ is a Vandermonde matrix, the coefficients $w(y_j)$ are uniquely determined by interpolation. In particular, they correspond to the coefficients of the Lagrange interpolation polynomial
\begin{equation}
   p(x) = \prod_{1\leq |j|\leq N}\frac{x-\lambda^j}{1-\lambda^j} 
   = \sum_{j=1}^{N_y} w(y_j) x^j, 
\end{equation}
which satisfies $p(\lambda^n) = \delta_{0}(n)$ for all $|n|\leq N$.

In analogy with the argument of \cite{moiola2009approximation}, one can show that the kernel function with the above weights $w$ provides a good approximation to $J_0(k|x-z|)$ in the following lemma. 

\begin{Lemma}
\label{approx_J0}
Assume the measurement aperture is $[-\alpha,\alpha]$ for $\alpha\leq \pi/2$ and let $\{y_j\}_{j=1}^{N_y}$ denote uniformly distributed measurement points on the boundary with $N_y = 2N+1$ and $N$ is large enough. For weights $w$ determined by \eqref{Vandermonde_matrix} and arbitrary $x,z \in \mathbb{R}^2$, the following error estimate holds:
\begin{equation}
\label{diff_Kw_J0}
\big|K_w(x,z)-J_0(k|x-z|)\big| 
=  \bigg|\sum_{j=1}^{N_y} e^{i k(x-z)\cdot y_j} w(y_j) - J_0(k|x-z|)\bigg| 
\leq \frac{(2N+1)}{(2\pi N)^{3/2}}
\left(\frac{ e^3 k|x-z|}{N \alpha^2}\right)^N.
\end{equation}
\end{Lemma}

\begin{proof}
By the definition of $w(y)$, set the difference
\[
E(|x-z|):=K_w(x,z)-J_0(k|x-z|).
\]
Using the Fourier expansion of $J_0$, we obtain
\[
E(|x-z|)
= \sum_{j=1}^{N_y} e^{ik (x-z)\cdot y_j} w(y_j) - J_0(k|x-z|)
= \sum_{|n|>N} e^{in(\theta_{x}-\theta_{z})} p(\lambda^n)J_n(k|x-z|) i^n \,,
\]
where $\lambda = e^{i d_{\theta}}$ with $d_{\theta} = 2\alpha/(N_y-1)$.  
Thus it remains to bound $E(|x-z|)$.

Recalling the Vandermonde matrix $V$ in \eqref{Vandermonde_matrix}, note that
\[
p(\lambda^n)=\sum_{j=1}^{N_y} w(y_j)\lambda^{jn} = v_n^T w,
\]
where $v_n(j)=\lambda^{jn}$. Hence
\[
|p(\lambda^n)| \le \|v_n\|_\infty \,\|V^{-1}\|_1 \,\|e\|_\infty
= \|V^{-1}\|_1,
\]
since $|\lambda|=1$ implies $\|v_n\|_\infty=1$ and $\|e\|_\infty=1$.

For $\|V^{-1}\|_1$ we recall the estimate from \cite{gautschi1978inverses}:
\begin{align}
\label{inv_V_1norm}
\|V^{-1}\|_1 &\leq \max_{-N\leq j \leq N} \prod_{\substack{-N\leq l \leq N \\ l\neq j}}
\frac{1+|\lambda^l|}{|\lambda^j-\lambda^l|} 
\leq \max_{-N\leq j \leq N} \prod_{\substack{-N\leq l \leq N \\ l\neq j}}
\frac{1}{|\sin(d_{\theta}(j-l))|} \notag \\
&\leq \prod_{1\leq l \leq N}\frac{1}{|\sin(d_{\theta} l )|^2} 
\leq \prod_{1 \leq l \leq N}\frac{\pi^2}{4(d_{\theta}l )^2}
\leq \frac{\pi^{2N}}{4^N d_{\theta}^{2N}(N!)^2}\,,
\end{align}
where we used the inequality $\sin t \ge \tfrac{2}{\pi}t$ for $t\in[0,\pi/2]$.

Next, to estimate the Bessel tail, recall the standard bound
\[
|J_n(t)|\leq \frac{(t/2)^n}{n!}\qquad (n\ge0,t\ge0).
\]
Hence
\begin{align}
\label{sum_Jn_n}
\sum_{|n|>N}|J_n(k|x-z|)| 
&\leq 2\sum_{n>N}\frac{(k|x-z|/2)^n}{n!}  \\
&\leq 2\frac{(k|x-z|/2)^{N+1}}{(N+1)!}\sum_{m\geq 0}\frac{(k|x-z|/2)^m}{m!} 
= \frac{2(k|x-z|/2)^{N+1}e^{k|x-z|/2}}{(N+1)!}. \notag
\end{align}

Applying Stirling's approximation
\begin{equation}
    \label{strirling}
\Big(\tfrac{n}{e}\Big)^n e^{1/(12n+1)}\sqrt{2\pi n} < n! < \Big(\tfrac{n}{e}\Big)^n e^{1/(12n)}\sqrt{2\pi n}\,,
\end{equation}
and combining \eqref{inv_V_1norm} and \eqref{sum_Jn_n}, we obtain
\begin{align}
|E(|x-z|)| 
&\leq \sum_{|n|>N}|J_n(k|x-z|)|\,|p(\lambda^n)| \notag \\
&\leq \frac{2(k|x-z|/2)^{N+1}e^{k|x-z|/2}}{(N+1)!}
\cdot \frac{\pi^{2N}(2N+1)}{4^N d_{\theta}^{2N}(N!)^2} \notag \\
&\leq \frac{(2N+1)}{(2\pi N)^{3/2}}
\left(\frac{e^3 k|x-z|}{N \alpha^2}\right)^N,
\end{align}
for $N$ sufficiently large, which yields the stated estimate.
\end{proof}

We emphasize that, for a fixed aperture, the weighted kernel $K_w(x,z)$ approximates the full-aperture kernel $J_0(k|x-z|)$ most accurately in two regimes: (i) when $|x-z|$ is small, so that higher-order Fourier modes have limited influence, and (ii) when the number of measurement points $N_y$ is sufficiently large, so that the condition \eqref{requirement_w} is enforced with high accuracy. In either case, the discrepancy between $K_w(x,z)$ and $J_0(k|x-z|)$ becomes small, and it is therefore reasonable to expect the approximation
\[
K_w(x,z) \approx J_0(k|x-z|)
\]
to hold, as stated in \eqref{K_W_approx}.

We now turn to the analysis of the index function defined in \eqref{def_index}.  
For clarity, we first consider the case where $D = B(0,\delta)$ for some $\delta > 0$, before extending the discussion to more general inclusions in $\mathbb{R}^2$.  
Using \eqref{def_U} together with the Jacobi-Anger identity \eqref{Jacobi_anger}, the discrete far-field operator in \eqref{def_far} can be written as
\begin{equation}
\label{far_field_matrix}
    F(y_i,y_j) 
    = \tilde{\gamma} \int_D 
      \Bigg(\sum_{n\in\mathbb{Z}} J_n(k|z|)\, e^{in(\theta_{y_i}-\theta_z)}\Bigg)
      \Bigg(\sum_{m\in\mathbb{Z}} J_m(k|z|)\, e^{im(\theta_{z}-\theta_{y_j})}\Bigg) dz 
    = \sum_{n\in\mathbb{Z}} e^{in(\theta_{y_i}-\theta_{y_j})}\, s_n ,
\end{equation}
where
\begin{equation}
\label{def_sn}
    s_n := \tilde{\gamma} \int_{B(0,\delta)} 
           \big| J_n(k|z|) e^{in\theta_z} \big|^2 dz   
         = \tilde{\gamma}\, 2\pi \int_{0}^{\delta} J_n(k r)^2 \, r\, dr
         = \tilde{\gamma}\, \pi \delta^2 \Big(J_n(k\delta)^2 - J_{n-1}(k\delta) J_{n+1}(k\delta)\Big).
\end{equation}
Here, the index $y_j$ in \eqref{far_field_matrix} refers to the incidence direction (denoted by $\boldsymbol{d}$ in \eqref{def_far}) which comes from the same discrete set as measurement points $y_i$. We adopt this notation for simplicity, since both incidence directions and measurement points are sampled from the same discrete set $\Gamma_{N_y}$.

\subsubsection{Approximate solution and error bounds}
\label{sec_error_sol}
With the coefficients $s_n$ in \eqref{def_sn} at hand, an approximate solution $\tilde{g}_x^w$ to the equation $F_w \tilde{g}_x^w = \eta_x$ can be constructed in the form
\begin{equation}\label{eq:newg}
    \tilde{g}_x^w(y) :=  
    \sum_{|m|\leq N/2} i^m \frac{J_m(k|x|)}{s_m}\, w(y)^{1/2} 
    e^{im(\theta_x-\theta_y)}\,,
\end{equation}
as established in the following lemma.
\begin{Lemma}
\label{Lemma_gw_approx}
Let $F$ denote the discrete far-field operator associated with the inhomogeneous medium $B(0,\delta)$ and the assumption in Lemma \ref{approx_J0} holds. Let $\tilde{g}_x^w$ and $\eta_x$ be the approximate solution and test functions defined in \eqref{eq:newg} and \eqref{def_gzw}, respectively. Suppose that the weight $w$ is chosen as the solution of \eqref{Vandermonde_matrix}, and assume $k\delta \leq \pi/2$. Then 
\begin{equation}
\label{Lemma_bound_structural}
    \|F_w \tilde{g}_x^w - \eta_x\|_2 
    \le\frac{e^{k|x|/2}}{2\pi N}\left[
   \frac{4 I_0(k\delta)}{(\pi N)^{1/2}}
   \left(\frac{\sqrt{2}\, e\, k\delta}{\alpha N}\right)^{2N}
   + \left(\frac{\sqrt{2}\, k|x| e^2}{\alpha N}\right)^N
\right].
\end{equation}
\end{Lemma}


\begin{proof}
Define the error term
\begin{equation}
E_2 := F_w \tilde{g}_x^w - \eta_x.
\end{equation}
By construction of the weights $w$ (see \eqref{Vandermonde_matrix}), we have 
\[
\sum_{j=1}^{N_y} e^{in\theta_{y_j}} w(y_j) = \delta_{0}(n), \qquad |n|\leq N,
\] 
and thus, for $y\in \Gamma_{N_y}$,
\begin{align}
\label{def_E2}
E_2(y) &= w^{1/2}(y)\Bigg[ 
  \sum_{|n|>N/2} i^n J_n(k|x|) e^{in(\theta_x-\theta_y)} 
+ \sum_{y_j}\sum_{\substack{|m|\leq N/2 \\ |n-m|>N}} 
   i^m \frac{s_n}{s_m} J_m(k|x|) e^{im\theta_y - in(\theta_y - \theta_{y_j})} w(y_j)\Bigg].
\end{align}
The first summation in \eqref{def_E2} coincides with the high-frequency remainder already estimated in \eqref{sum_Jn_n}.

We now focus on the second summation in \eqref{def_E2}. The challenge here is the ratio $s_n/s_m$. By symmetry $s_n=s_{-n}$, so it suffices to treat the case $n > N/2 > m > 0$. For $k\delta < \pi/2$, the classical bounds on ratios of Bessel functions \cite{bessel_inequalities_1990} yield
\begin{equation}
\frac{k\delta}{2(n+1)} \;\leq\; \frac{J_{n+1}(k\delta)}{J_n(k\delta)} \;\leq\; \frac{k\delta}{2n+1}.
\end{equation}
From the definition of $s_n$ in \eqref{def_sn}, this implies
\begin{equation}
\label{snsm}
\frac{s_n}{s_m} 
= \frac{J_n(k\delta)^2\Big(1 - \tfrac{J_{n-1}(k\delta)}{J_n(k\delta)} \tfrac{J_{n+1}(k\delta)}{J_n(k\delta)}\Big)}{J_m(k\delta)^2\Big(1 - \tfrac{J_{m-1}(k\delta)}{J_m(k\delta)} \tfrac{J_{m+1}(k\delta)}{J_m(k\delta)}\Big)}
\;\leq\; \frac{J_n(k\delta)^2}{J_m(k\delta)^2}\,\frac{m+1/2}{n+1}.
\end{equation}

For the ratio $J_n(k\delta)/J_m(k\delta)$, we use the inequality from \cite{inequal_special} together with the monotonicity of the Gamma function:
\begin{equation}
\label{JnJm}
\frac{J_n(k\delta)}{J_m(k\delta)} 
\;\leq\; \left(\frac{k\delta}{2}\right)^{n-m} \frac{\Gamma(m+1/2)}{\Gamma(n+1/2)}
\;\leq\; \left(\frac{k\delta}{2}\right)^{n-m} \sqrt{\frac{m}{n}} \frac{m!}{n!}.
\end{equation}
Combining \eqref{snsm} and \eqref{JnJm}, and applying standard Bessel bounds, we obtain
\begin{equation}
\frac{s_n}{s_m}\, J_m(k|x|) 
\;\leq\; \left(\frac{k\delta}{2}\right)^{2(n-m)} \left(\frac{k|x|}{2}\right)^m \frac{1}{(n-m)!^2\, m!}.
\end{equation}

Summing over the indices and using the definition of the modified Bessel function $I_0$
\[
\sum_{n-m\geq N} \frac{(k\delta/2)^{2(n-m)}}{(n-m)!^2} 
\;\leq\; I_0(k\delta)\, \frac{(k\delta/2)^{2N}}{(N!)^2},
\]
we find
\begin{align}
\label{E2_2}
\sum_{\substack{|m|\leq N/2 \\ |n-m|\geq N}} \Big|\tfrac{s_n}{s_m} J_m(k|x|)\Big|
&\leq 4 I_0(k\delta) \sum_{0\leq m \leq N/2} \frac{(k|x|/2)^m}{m!} \frac{(k\delta/2)^{2N}}{(N!)^2} \\
&\leq 4 I_0(k\delta)\, e^{k|x|/2} \frac{(k\delta/2)^{2N}}{(N!)^2}. \notag
\end{align}

Moreover, by \eqref{inv_V_1norm}, the weights satisfy
\[
\|w\|_\infty \;\leq\; \frac{1}{2\pi N}\left(\frac{\sqrt{2}\, e}{\alpha}\right)^{2N}.
\]
Combining this bound with the high-frequency estimate \eqref{sum_Jn_n} and the second-term estimate \eqref{E2_2}, we obtain
\begin{align}
\|E_2\|_2 
&\leq \|w^{1/2}\|_\infty \Bigg[
   4 I_0(k\delta)\, e^{k|x|/2} \frac{(k\delta/2)^{2N}}{(N!)^2} 
   + \frac{e^{k|x|/2}}{\sqrt{2\pi N}}
     \left(\frac{k|x| e}{2N}\right)^N
   \Bigg] \\
&\leq \frac{e^{k|x|/2}}{2\pi N}\left[
   \frac{4 I_0(k\delta)}{(\pi N)^{1/2}}
   \left(\frac{\sqrt{2}\, e\, k\delta}{\alpha N}\right)^{2N}
   + \left(\frac{\sqrt{2}\, k|x| e^2}{\alpha N}\right)^N
\right]. \notag
\end{align}
This completes the proof.
\end{proof}

\noindent
\textbf{Remark.}  
The error $E_2$ in \eqref{def_E2} decays super-algebraically, being of order $(1/N)^N$. Thus, the approximate solution $\tilde{g}_x^w$ from \eqref{eq:newg} is extremely accurate for moderately large $N$. 

We are now prepared to compute the $L^2$-norm of the approximate solution, which defines the index function:
\begin{equation}
\label{norm_I}
\tilde{I}_w^{-1}(x) := \|\tilde{g}_x^w\|_2^{-2} = \sum_{|n|<N/2}\frac{J_n(k|x|)^2}{s_n^2}\,.
\end{equation}
It follows directly from \eqref{norm_I} that $\tilde{I}_w(x)$ coincides with the index function constructed from full-aperture data sampled at $(N_y-1)/2+1$ uniformly distributed measurement points on $\bS$.

To assess the behavior of the index function over a region, we define the spatial average of $\tilde{I}_w^{-1}$ on a set $\Omega$ as
\begin{equation}
    \overline{\|\tilde{I}_w^{-1}\|}_{L^1(\Omega)} := \frac{1}{|\Omega|}\int_{\Omega} \big|\tilde{I}_w^{-1}(x)\big|\,dx\,.
\end{equation}
With this definition, we show that $\tilde{I}_w^{-1}(x)$ is very small inside $B(0,\delta)$ but becomes large outside this region.

\begin{Th}
\label{Thm_concen}
Let $\tilde{I}_w^{-1}$ be defined by \eqref{norm_I} for $D=B(0,\delta)$ with $k\delta<\pi/2$ and the assumption in Lemma \ref{approx_J0} holds. Denote $n_0:=\lfloor N/2\rfloor$ and assume $r>\delta$. Then there exists a constant $C>0$ that is independent of $N$ such that
\begin{equation}
\label{fraction_index}
    \frac{\overline{\|\tilde{I}_w^{-1}\|}_{L^1(B(0,r))}}{\overline{\|\tilde{I}_w^{-1}\|}_{L^1(B(0,\delta))}}
    \;\ge\; C\,\frac{(r/\delta)^{2n_0}}{N} \quad\longrightarrow\; \infty
    \quad\text{as }N\to\infty.
\end{equation}
In particular, for large $N$ the ratio grows at least exponentially in $N$ (since $(r/\delta)^{2n_0}\sim (r/\delta)^N$).
\end{Th}

\begin{proof}
Recall the definition
\[
\tilde I_w^{-1}(x)=\sum_{|n|<N/2}\frac{J_n(k|x|)^2}{s_n^2},
\qquad s_n=\tilde\gamma\,2\pi\int_0^\delta J_n(kr)^2 r\,dr.
\]
Integrating in polar coordinates gives, for any $\rho>0$,
\begin{equation}
\int_{B(0,\rho)}\tilde I_w^{-1}(x)\,dx
= \sum_{|n|<N/2}\frac{1}{s_n^2}\,2\pi\int_0^\rho J_n(kr)^2 r\,dr
= \frac{1}{\tilde\gamma}\sum_{|n|<N/2}\frac{\displaystyle\int_0^\rho J_n(kr)^2 r\,dr}
{\displaystyle\int_0^\delta J_n(kr)^2 r\,dr}.
\end{equation}
Thus the ratio of spatial averages equals the ratio of the corresponding sums, since the common factor $1/\tilde\gamma$ cancels.

Fix the single index $n_0=\lfloor N/2\rfloor$. All summands are nonnegative, so
\begin{equation}
\frac{\int_{B(0,r)}\tilde I_w^{-1}}{\int_{B(0,\delta)}\tilde I_w^{-1}}
\;\ge\; \frac{\displaystyle\frac{\int_0^r J_{n_0}(kr)^2 r\,dr}{\int_0^\delta J_{n_0}(kr)^2 r\,dr}}
{\displaystyle\sum_{|n|<N/2}\frac{\int_0^r J_n(kr)^2 r\,dr}{\int_0^\delta J_n(kr)^2 r\,dr}}
\ge \frac{\int_0^r J_{n_0}(kr)^2 r\,dr}{(2n_0+1)\int_0^\delta J_{n_0}(kr)^2 r\,dr}.
\end{equation}
We now estimate the ratio of the radial integrals for index \(n_0\). Use the identity
\[
\int_0^\rho J_n(kr)^2 r\,dr
= \frac{\rho^2}{2}\Big(J_n(k\rho)^2 - J_{n-1}(k\rho)J_{n+1}(k\rho)\Big).
\]
Hence
\begin{equation}
\frac{\int_0^r J_{n_0}(kr)^2 r\,dr}{\int_0^\delta J_{n_0}(kr)^2 r\,dr}
= \frac{r^2\big(J_{n_0}(kr)^2 - J_{n_0-1}(kr)J_{n_0+1}(kr)\big)}
{\delta^2\big(J_{n_0}(k\delta)^2 - J_{n_0-1}(k\delta)J_{n_0+1}(k\delta)\big)}.
\end{equation}

To control the correction factors, use the standard ratio bound valid for all \(n\ge0\) and \(t\ge0\):
\[
\frac{|J_{n\pm1}(t)|}{|J_n(t)|} \le \frac{t}{2n+1}.
\]
Applying this with \(t=k r\) and \(t=k\delta\) yields
\begin{equation}
\left| \frac{J_{n_0-1}(k\rho)J_{n_0+1}(k\rho)}{J_{n_0}(k\rho)^2} \right|
\le \left(\frac{k\rho}{2n_0+1}\right)^2 \qquad(\rho=r,\delta).
\end{equation}
Since \(k\delta\le \pi/2\) and \(n_0\ge 1\) for large \(N\), the right-hand side is uniformly small and bounded by a constant \(<1/2\) for \(N\) large enough. Consequently there exist absolute constants \(c_-,c_+>0\) that are independent of \(n_0,N\) such that
\[
c_- J_{n_0}(k\rho)^2 \le J_{n_0}(k\rho)^2 - J_{n_0-1}(k\rho)J_{n_0+1}(k\rho)
\le c_+ J_{n_0}(k\rho)^2
\qquad(\rho=r,\delta).
\]
Therefore
\begin{equation}
\frac{\int_0^r J_{n_0}(kr)^2 r\,dr}{\int_0^\delta J_{n_0}(kr)^2 r\,dr}
\ge \frac{c_-}{c_+}\,\frac{r^2 J_{n_0}(kr)^2}{\delta^2 J_{n_0}(k\delta)^2}.
\end{equation}

Finally use the large-order estimate for fixed \(t\) and large \(n\),
\[
J_n(t)\sim \frac{(t/2)^n}{n!},
\]
which implies for \(n=n_0\) and fixed \(k,r,\delta\) 
\[
\frac{J_{n_0}(kr)}{J_{n_0}(k\delta)} \;\gtrsim\; \Big(\frac{r}{\delta}\Big)^{n_0}.
\]
Combining the previous displays yields
\begin{equation}
\frac{\int_{B(0,r)}\tilde I_w^{-1}}{\int_{B(0,\delta)}\tilde I_w^{-1}}
\;\gtrsim\; \frac{1}{2n_0+1}\left(\frac{r}{\delta}\right)^{2n_0+2}.
\end{equation}
Using \(2n_0+1\le N\) and \(2n_0+2\ge N-2\) gives the claimed bound \eqref{fraction_index}, and the right-hand side tends to infinity exponentially in \(N\) when \(r>\delta\). This completes the proof.
\end{proof}

From Theorem~\ref{Thm_concen}, since $r/\delta > 1$, the mean value of $\tilde{I}_w^{-1}$ is significantly larger in $B(0,r) \setminus B(0,\delta)$ than inside $B(0,\delta)$ as $N_y$ increases.  
Equivalently, the index function $\tilde{I}_w$ is large in $B(0,\delta)$ and small outside, enabling us to reconstruct the support of the inhomogeneous medium $D$.  

Finally, we compare the approximate solution $\tilde{g}_x^w$ in \eqref{eq:newg} with the exact solution $g_x^w$ of the far-field equation defined in \eqref{def_index}.  
To this end, we study the stability of the discrete operator $F_w$ and, in particular, conditions under which $\|F_w^{-1}\|_{2}$ remains bounded as the number of measurement points increases.  

Let  
\[
\langle u, v \rangle_{w,N_y} = \sum_{j=1}^{N_y} u_j \overline{v_j}\, w_j
\]
denote the weighted inner product on the measurement nodes $\Gamma_{N_y} = \{ y_j \}$, and define the orthogonal projector  
\[
P_{N_y}: \mathbb{C}^{N_y} \to V_{N_y} = \mathrm{span}\{(e^{in\theta_{y_j}})_{j=1}^{N_y} : |n|<N_y \}
\]
in $\ell^2_{w,N_y}$. By construction,
\begin{equation}
    F_w = P_{N_y} F P_{N_y}\,.
\end{equation}
Since $P_{N_y}$ is an orthogonal projection, the projection theorem yields
\begin{equation}
\lim_{N_y \to \infty} \|F - F_w\|_{2} = 0\,.
\end{equation}
Consequently, for $N_y$ sufficiently large,
\[
\|(F_w - F)F^{-1}\|_{2} < 1\,.
\]
Writing
\[
F_w = F(I - E), \qquad E = F^{-1}(F - F_w)\,,
\]
we obtain $(I - E)^{-1} = \sum_{k=0}^\infty E^k$ and therefore
\begin{equation}
F_w^{-1} = (I - E)^{-1}F^{-1}\,, 
\qquad
\|F_w^{-1}\|_{2} \leq \frac{\|F^{-1}\|_{2}}{1 - \|(F_w - F)F^{-1}\|_{2}}.
\end{equation}

This shows that $\|F_w^{-1}\|_{2}$ remains bounded for large $N_y$, provided that $F$ is invertible and $\|(F_w - F)F^{-1}\|_{2}$ is sufficiently small.  
We note, however, that the rate at which $F_w \to F$ depends on the smoothness of the far-field operator $F$, and thus ultimately on the smoothness of the inhomogeneous medium.  
In highly ill-conditioned regimes (e.g., with very limited aperture or strong measurement noise), the stability of $F_w$ may deteriorate, and no uniform bound can be guaranteed.

We next quantify the difference between the exact solution \(g_x^w\) of the weighted far-field equation and the approximate solution \(\tilde g_x^w\) constructed in \eqref{eq:newg}. 

Using \eqref{def_sn} we obtain the following lower bound for \(\|\tilde g_x^w\|_2\):
\begin{align}
    \|\tilde g_x^w\|_2^2 
    &= \sum_{|n|<N/2} \frac{J_n(k|x|)^2}{s_n^2}
    \;\ge\; \sum_{|n|<N/2} \frac{J_n(k|x|)^2 (n+1/2)^2}{J_n(k\delta)^4} \notag\\
    &\ge\; \sum_{|n|<N/2} \left(\frac{2|x|}{k\delta^2}\right)^{2n} \frac{n!^2 (n+1)^2}{4}\,,
\end{align}
where in the last step we used the standard large-order lower bound \(J_n(t)\gtrsim (t/2)^n/n!\) for fixed \(t\) and large \(n\).

Recalling the error estimate for \(E_2\) in Lemma~\ref{Lemma_gw_approx}, we have
\[
\|g_x^w - \tilde g_x^w\|_2 = \|F_w^{-1} E_2\|_2 \le \|F_w^{-1}\|_2 \, \|E_2\|_2\,,
\]
and therefore the relative error admits the bound
\begin{equation}
\label{eq:rel_error_bound_final}
\frac{\|g_x^w - \tilde g_x^w\|_2}{\|\tilde g_x^w\|_2}
\le \|F_w^{-1}\|_2\;
\frac{\|E_2\|_2}
{\bigg(\displaystyle\sum_{|n|<N/2}\left(\frac{2|x|}{k\delta^2}\right)^{2n}\frac{n!^2 (n+1)^2}{4}\bigg)^{1/2}}\,.
\end{equation}

Because the denominator grows factorially in \(n\) (and hence in \(N\)) while the numerator grows only exponentially in \(N\), the right-hand side of \eqref{eq:rel_error_bound_final} decays exponentially in \(N\). In particular, for every fixed compact set \(K\subset\mathbb R^2\), e.g., \(K=B(0,r)\) or \(K=B(0,\delta)\), we obtain
\[
\varepsilon_K(N):=\sup_{x\in K}\frac{\|g_x^w-\tilde g_x^w\|_2}{\|\tilde g_x^w\|_2}\longrightarrow 0
\qquad (N\to\infty)\,,
\]
provided \(\|F_w^{-1}\|_2\) remains uniformly bounded.

\begin{Cor}
\label{Cor_concen_true}
Under the hypotheses of Theorem~\ref{Thm_concen} and assume $\|F_w^{-1}\|_2$ is uniformly bounded, the true index \(I_w^{-1}\) satisfies the same concentration property as \(\tilde I_w^{-1}\). Concretely, for any \(r>\delta\),
\[
\frac{\overline{\|I_w^{-1}\|}_{L^1(B(0,r))}}{\overline{\|I_w^{-1}\|}_{L^1(B(0,\delta))}}
\longrightarrow \infty \qquad (N\to\infty).
\]
\end{Cor}

\begin{proof}
Set
\[
\varepsilon(x;N):=\frac{\|g_x^w-\tilde g_x^w\|_2}{\|\tilde g_x^w\|_2},\qquad
\varepsilon_{B(0,r)}(N)=\sup_{x\in B(0,r)}\varepsilon(x;N),\qquad
\varepsilon_{B(0,\delta)}(N)=\sup_{x\in B(0,\delta)}\varepsilon(x;N)\,.
\]
From the estimate \eqref{eq:rel_error_bound_final} we have \(\varepsilon_{B(0,r)}(N)\), \(\varepsilon_{B(0,\delta)}(N)\to0\) as \(N\to\infty\).

The triangle inequality yields
\[
\big|\|g_x^w\|_2-\|\tilde g_x^w\|_2\big|\le \|g_x^w-\tilde g_x^w\|_2
\le \varepsilon(x;N)\,\|\tilde g_x^w\|_2\,,
\]
hence
\[
(1-\varepsilon(x;N))\|\tilde g_x^w\|_2 \le \|g_x^w\|_2 \le (1+\varepsilon(x;N))\|\tilde g_x^w\|_2\,.
\]
Using the definition of index functions gives the pointwise inequalities
\[
(1+\varepsilon(x;N))^{-2}\,\tilde I_w^{-1}(x)\le I_w^{-1}(x)\le(1-\varepsilon(x;N))^{-2}\,\tilde I_w^{-1}(x).
\]

Integrating these inequalities over a measurable set \(A\) and dividing by \(|A|\) yields the corresponding bounds for the spatial averages. Applying this to \(A=B(0,r)\) and \(A=B(0,\delta)\) and taking the lower bound for the numerator together with the upper bound for the denominator, we obtain
\[
\frac{\overline{\|I_w^{-1}\|}_{L^1(B(0,r))}}{\overline{\|I_w^{-1}\|}_{L^1(B(0,\delta))}}
\ge
\frac{(1+\varepsilon_{B(0,r)}(N))^{-2}}{(1-\varepsilon_{B(0,\delta)}(N))^{-2}}
\cdot
\frac{\overline{\|\tilde I_w^{-1}\|}_{L^1(B(0,r))}}{\overline{\|\tilde I_w^{-1}\|}_{L^1(B(0,\delta))}}\,.
\]
The first term on the right tends to \(1\) as \(N\to\infty\) because \(\varepsilon_{B(0,r)}(N)\), \(\varepsilon_{B(0,\delta)}(N)\to0\). Combining this with Theorem~\ref{Thm_concen}  yields the asserted divergence for the true index \(I_w^{-1}\).
\end{proof}

\subsubsection{Extension beyond single disk}
\label{sec_gen_analysis}
We shall briefly sketch the extension of the results for the index function, originally developed for reconstructing $B(0,\delta)$, to a more general inhomogeneous medium in $\mathbb{R}^2$. 

\begin{enumerate}
\item \textbf{Translation invariance for a disk.}
Let $B(z,\delta)$ be an arbitrary disk and write $F_0$ and $F_z$ for the discrete far-field operators associated with $B(0,\delta)$ and $B(z,\delta)$, respectively.  Under the discrete Born approximation \eqref{def_U} (and absorbing any constant phase factors into the definition of the testing vectors), the matrix entries for $F_z$ and $F_0$ are related by the pointwise phase shift
\[
\Psi_z(y_j) = e^{-ik z\cdot y_j},\qquad y_j\in\Gamma_{N_y},
\]
so that, in matrix form,
\[
F_z^* F_z \;=\; \operatorname{diag}(\overline{\Psi_z})\, F_0^* F_0 \,\operatorname{diag}(\Psi_z)\,,
\]
where $\operatorname{diag}(\Psi_z)$ denotes the diagonal matrix with entries $\Psi_z(y_j)$ which is unitary and invertible.

Let $I_w(x;B(z,\delta))$ denote the index function computed from $F_z$. Using $\Psi_{x+z}=\operatorname{diag}(\Psi_z)\Psi_x$ we obtain
\begin{align}
I_w^{-1}(x;B(z,\delta))
&= \Psi_x^*(F_z^*F_z)^{-1}\Psi_x \notag\\
&= \Psi_{x+z}^*\,\operatorname{diag}(\Psi_z)\,
\bigl(\operatorname{diag}(\overline{\Psi_z})F_0^*F_0\operatorname{diag}(\Psi_z)\bigr)^{-1}
\operatorname{diag}(\overline{\Psi_z})\,\Psi_{x+z} \notag\\
&= \Psi_{x+z}^*(F_0^*F_0)^{-1}\Psi_{x+z}
\;=\; I_w^{-1}(x+z;B(0,\delta)). \label{index_translate}
\end{align}
In words: the index for a translated disk $B(z,\delta)$ is exactly the translation (by $-z$) of the index for the centered disk $B(0,\delta)$. Hence the concentration result of Theorem~\ref{Thm_concen} for $B(0,\delta)$ carries over immediately to any translated disk $B(z,\delta)$.

\item \textbf{Extensions to unions of disks and general inclusions.}
The proofs of Lemma~\ref{Lemma_gw_approx} and Theorem~\ref{Thm_concen} are based on the Fourier expansion of the kernel
\[
\int_D e^{ik(y_i-y_j)\cdot z}\,dz
\]
on the measurement arc $\Gamma$, expressed in the Fourier modes $e^{in(\theta_{y_i}-\theta_{y_j})}$. For a general domain $D$ the same strategy applies whenever one can compute or accurately estimate the corresponding Fourier coefficients
\[
s_n(D) \;=\; \int_D e^{ik(y-y')\cdot z}\,dz  \,.
\]
If these coefficients are available (analytically or numerically), one can build the approximate solution \eqref{eq:newg} and repeat the concentration argument to obtain an analogue of Theorem~\ref{Thm_concen} for that domain.

We sketch the extension to unions of well-separated inclusions. While the analysis is not fully rigorous, the main localization mechanism carries over. Assume $D$ consists of $M$ well-separated disks $B(z_m,\delta_m)$, $m=1,\dots,M$. Then the far-field operator splits into a superposition of contributions from each disk, and its Fourier coefficients take the convolutional form
\begin{equation}
\label{multi_disk_coeffs}
\tilde s_n \;=\; \sum_{m=1}^M \sum_{t\in\mathbb Z} s_t^{(m)}\, g_{\,n-t}^{(m)},
 \quad\text{where}\quad
\begin{cases}
    s_t^{(m)} \;:=\; \int_{B(0,\delta_m)} J_t(k|z|)^2 \,dz\\
    g_n^{(m)} \;:=\; \int_{\mathbb S^{1}} e^{in\theta}\, e^{ik y\cdot z_m}\,d\sigma_y
\end{cases}
\end{equation}
Consequently, the discrete far-field matrix admits the Fourier-series representation
\[
F(y_i,y_j) \;=\; \sum_{n\in\mathbb Z} e^{in(\theta_{y_i}-\theta_{y_j})}\,\tilde s_n,
\]
with $\tilde s_n$ as in \eqref{multi_disk_coeffs}. In this situation, one can construct the approximate solution \(\tilde g_x^w\) by truncating the Fourier expansion and choosing the weights as before; the same dominant-mode / diagonal-dominance arguments then yield a concentration result for the reconstructed index, with each disk producing a localized peak where the interaction terms can be controlled when the disks are sufficiently separated.

\end{enumerate}

To avoid unnecessary technicalities, we defer a full theoretical treatment for arbitrary domains to future work; instead, in Section~\ref{sec_numerics} we verify the method numerically on a variety of examples, including point scatterers, rectangular bars, and a non-convex kite-shaped inclusion, and observe that the weighted index retains the same localization properties in practice.


 \section{Computation of the weight function and stability of the reconstruction}
\label{sec_construction_w}

In this section, we compute the numerical values of the weight function $w$ and examine the robustness of the reconstruction in the presence of random noise. A direct evaluation of $w$ via the Vandermonde system in \eqref{Vandermonde_matrix} is generally unstable and requires solving a large linear system of size $N_z \times N_y$. To address this issue, we develop a more stable and efficient method for computing $w$, which is independent of the discretization of $\Omega$ and hence does not depend on $N_z$.

\subsection{Computation of the weight function}
To simplify the presentation, we consider the sampling domain $\Omega = B(0,1)$ and a set of discrete sampling points 
\[
\Omega_{N_z} := \{z_n\}_{n=1}^{N_z} \subset \Omega.
\]
Recalling the definition of $K_w$ in \eqref{def_Kw}, we note the translation invariance property 
\[
K_w(x,z) = K_w(0, x-z)\,.
\]
Hence, in order to construct a suitable weight function $w$ satisfying \eqref{approx_Kw}, it suffices to enforce the condition
\begin{equation}
\label{Uw_v}
   K_w(0,z) = \sum_{j=1}^{N_y} e^{ik y_j \cdot z}\, w(y_j) 
   \;\approx\; 
   \begin{cases}
       J_0(k|z|), & \text{in } \mathbb{R}^2, \\[0.3em]
       j_0(k|z|), & \text{in } \mathbb{R}^3,
   \end{cases}
\end{equation}
for $z \in \Omega_{N_z}$. Equation \eqref{Uw_v} can be written in the compact matrix form
\[
Uw = v\,,
\]
where $U \in \mathbb{C}^{N_z \times N_y}$ has $j$-th column given by $\big(e^{ik z_n \cdot y_j}\big)_{n=1}^{N_z}$, the vector $w \in \mathbb{C}^{N_y}$ contains the unknown weights, and the right-hand side is defined by $v(n) = J_0(k|z_n|)$ in $\mathbb{R}^2$ or $v(n) = j_0(k|z_n|)$ in $\mathbb{R}^3$, for $n=1,\dots,N_z$. We emphasize that the notation $U$ is similar to that in \eqref{def_U}, as both matrices are constructed with exponential terms, but here the sampling set $\{z_n\}$ is from the domain of interest $\Omega$.

We first describe the computation of $w$ in the two-dimensional case $\mathbb{R}^2$. A direct approach to solving $Uw=v$ is to compute an approximate inverse of $U$ using truncated singular value decomposition (SVD). Specifically, let $U^{\dagger}$ denote the truncated SVD pseudoinverse of $U$, then
\[
w = U^{\dagger} v\,.
\]
As an illustrative example, consider the aperture $\alpha = \pi/4$, wavenumber $k=6$, and target function $v(z) = J_0(k|z|)$. Solving $Uw=v$ via truncated SVD using the first $16$ singular values yields the weight function $w(y)$, which is plotted in Fig.\,\ref{fig_w}.

\begin{figure}[h!] 
  \captionsetup{width=.9\linewidth}
  \centering
  \includegraphics[trim={1.5cm 8.3cm 1.5cm 9cm},clip,scale = 0.5]{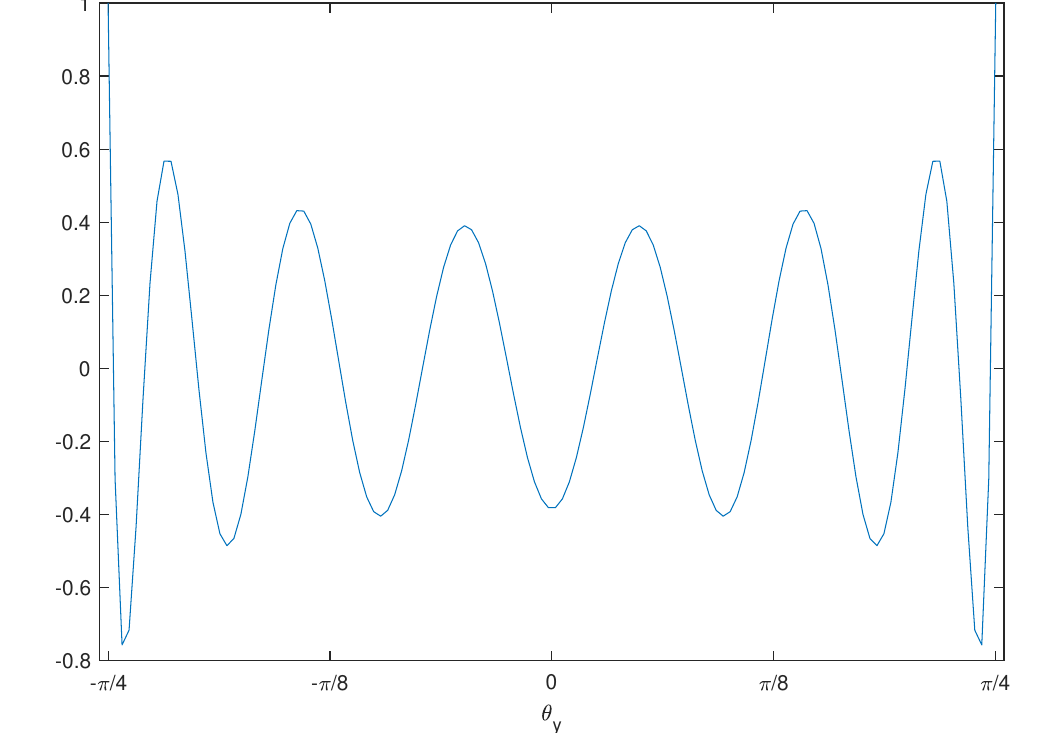}
  \caption{Computed values of $w(y)$ obtained by solving \eqref{Uw_v} via truncated SVD using the first $16$ singular values. Here $y = [\cos(\theta), \sin(\theta)]$ with $\theta \in [-\pi/4, \pi/4]$, $k=6$, and $v(z) = J_0(k|z|)$.}
\label{fig_w}
\end{figure}
From Fig.\,\ref{fig_w}, it is apparent that the weight function $w$ can be well approximated by a linear combination of Fourier modes of the form $\cos(n \theta_y)$, with $n = 0,\dots, M$. Choosing $M < N_y$ ensures that $w$ is sufficiently smooth, which in turn improves the stability of computing the index function. Moreover, this approximation substantially reduces the dimension of the system to be solved for $w$. It is worth emphasizing that solving \eqref{Uw_v} directly in its full form is generally computationally inefficient: the system matrix $U$ has size $N_z \times N_y$, and in practice $N_z \gg N_y$. The Fourier-mode reduction therefore provides both stability and computational efficiency.

Our first structural assumption on the weight function $w$ is the symmetry condition 
\[
w(y) = w(-y)\,,
\] 
which follows naturally from the symmetry of the $\Gamma$ in \eqref{def_Gamma} when $y = (\cos\theta, \sin\theta)$ with $\theta \in [-\alpha,\alpha]$. Motivated by the numerical observations in Fig.\,\ref{fig_w}, we represent $w$ in terms of a truncated Fourier–cosine expansion of the form
\[
w(y) = \sum_{m=0}^M \beta_m v_m(y)\,,
\]
where the basis functions are defined by
\begin{equation}
\label{def_vm}
    v_m(y(\theta)) := \frac{1}{\sqrt{\alpha}} \cos\!\left(\frac{\pi m}{\alpha}\theta\right)\,.
\end{equation}
Here $\{\beta_m\}_{m=0}^M$ are the unknown coefficients to be determined. 

To compute $w$, it suffices to solve the system
\[
UP_M \beta = v\,,
\]
where $P_M \in \mathbb{C}^{N_y \times (M+1)}$ is the basis matrix with entries $P_M(j,m) = v_m(y(\theta_j))$, $\beta \in \mathbb{C}^{M+1}$ is the coefficient vector with $\beta(m) = \beta_m$, and $v$ is defined as in \eqref{Uw_v}. 

Since the system is typically underdetermined when $M < N_z$, we compute $w$ by solving the associated normal equation
\begin{equation}
\label{normal_equation}
(U P_M)^* (U P_M) \beta = (U P_M)^* v\,.
\end{equation}
The matrix on the left-hand side of \eqref{normal_equation} is Hermitian and positive definite of size $(M+1)\times(M+1)$. Based on numerical experiments, choosing $M \approx (N_y-1)/2$ provides sufficient accuracy in the approximation \eqref{approx_Kw}, while ensuring stability.  
This approach reformulates the ill-conditioned Vandermonde system \eqref{Vandermonde_matrix} as a smaller Hermitian system, thereby reducing the cost of the linear solver for dense matrix from $\mathcal{O}(N_y^3)$ to $\mathcal{O}\!\left((N_y/2)^3\right)$, i.e., by roughly a factor of $8$.

To make \eqref{normal_equation} explicit, we rewrite the matrix–vector products in integral form. First, the action of $U P_M$ on $\beta$ can be expressed as
\begin{equation}
   (U P_M \beta)(z) = \frac{1}{\sqrt{\alpha}} \sum_{m=0}^M \beta_m 
   \int_{-\alpha}^{\alpha} e^{ik z \cdot y} \cos\!\left(\frac{\pi m}{\alpha} y\right) \, d\sigma_y\,.
\end{equation}
Next, for any $f \in L^2(\Omega)$, the adjoint operator satisfies
\begin{equation}
  \big((U P_M)^* f\big)(m) 
  = \frac{1}{\sqrt{\alpha}} \int_{\Omega} \int_{-\alpha}^{\alpha} 
  e^{-ik z \cdot y} \cos\!\left(\frac{\pi m}{\alpha} y\right) d\sigma_y \, f(z)\, dz,
  \quad m=0,\dots,M\,.
\end{equation}

Combining these expressions, the matrix entries of the left-hand side of \eqref{normal_equation} are given by
\begin{align}
\label{LHS_w}
    (UP_M)^*(UP_M)(m,n) 
     &= \frac{1}{\alpha}\int_{-\alpha}^{\alpha} \left[ \int_{-\alpha}^{\alpha} \int_{\Omega} 
     e^{ik z \cdot (y-y')} \, dz \, 
     \cos\!\left(\tfrac{\pi n}{\alpha} y\right) d\sigma_y \right] 
     \cos\!\left(\tfrac{\pi m}{\alpha} y'\right) d\sigma_{y'} \\
     &= \frac{2\pi}{\alpha}\int_{-\alpha}^{\alpha} \left[ \int_{-\alpha}^{\alpha} 
     \frac{J_1(k|y-y'|)}{k|y-y'|} 
     \cos\!\left(\tfrac{\pi n}{\alpha} y\right) d\sigma_y \right] 
     \cos\!\left(\tfrac{\pi m}{\alpha} y'\right) d\sigma_{y'}\,, \notag
\end{align}
where the simplification uses the identity
\[
\int_{\Omega} e^{ik z \cdot (y-y')} \, dz 
= 2\pi \int_0^1 J_0(k|z|\,|y-y'|)\,|z|\,dz 
= \frac{2\pi}{k|y-y'|} J_1\!\big(k|y-y'|\big)\,.
\]

For the right-hand side of \eqref{normal_equation}, if $v(z) = J_0(k|z|)$ we obtain
\begin{align}
\label{RHS_w}
     \big((UP_M)^*v\big)(m) 
     &= \frac{1}{\sqrt{\alpha}} \int_{\Omega} \int_{-\alpha}^{\alpha} 
     e^{-ik y \cdot z} \cos\!\left(\tfrac{\pi m}{\alpha} y\right) d\sigma_y \, J_0(k|z|) \, dz \\
     &= \frac{2 \pi}{\sqrt{\alpha}} \int_0^1 J_0(k|z|)^2 \, |z|\, d|z| \; \delta_0(m)\,. \notag
\end{align}

Thus, combining \eqref{LHS_w} and \eqref{RHS_w}, the computation of $w$ reduces to solving a linear system of size $(M+1)\times(M+1)$, where each entry is explicitly represented in terms of integrals involving Bessel functions. This formulation ensures both stability and computational efficiency.

We now carry out numerical experiments to compare 
$\langle \Psi_x, w \Psi_z \rangle_{L^2(\Gamma)}$, 
$\langle \Psi_x, \Psi_z \rangle_{L^2(\Gamma)}$, 
and $\langle \Psi_x, \Psi_z \rangle_{L^2(\bS)}$, 
for $x \in [-0.5,0.5]^2$, $z = (0.2,0)$, with $N_y = 12$ and $64$. The results are presented in Fig.\,\ref{fig_sim}.
\begin{figure}[H]
    \centering
    \begin{subfigure}[b]{0.3\textwidth}
        \centering
        \includegraphics[trim={3cm 9cm 3.2cm 9cm},clip,scale=0.38]{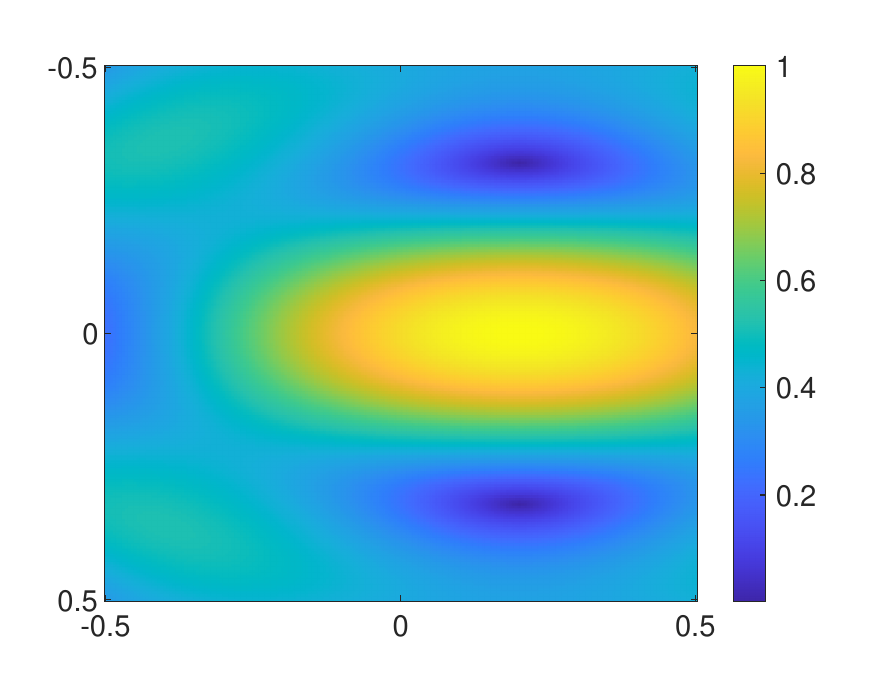}
    \end{subfigure}
    \begin{subfigure}[b]{0.3\textwidth}
        \centering
        \includegraphics[trim={3cm 9cm 3.2cm 9cm},clip,scale=0.38]{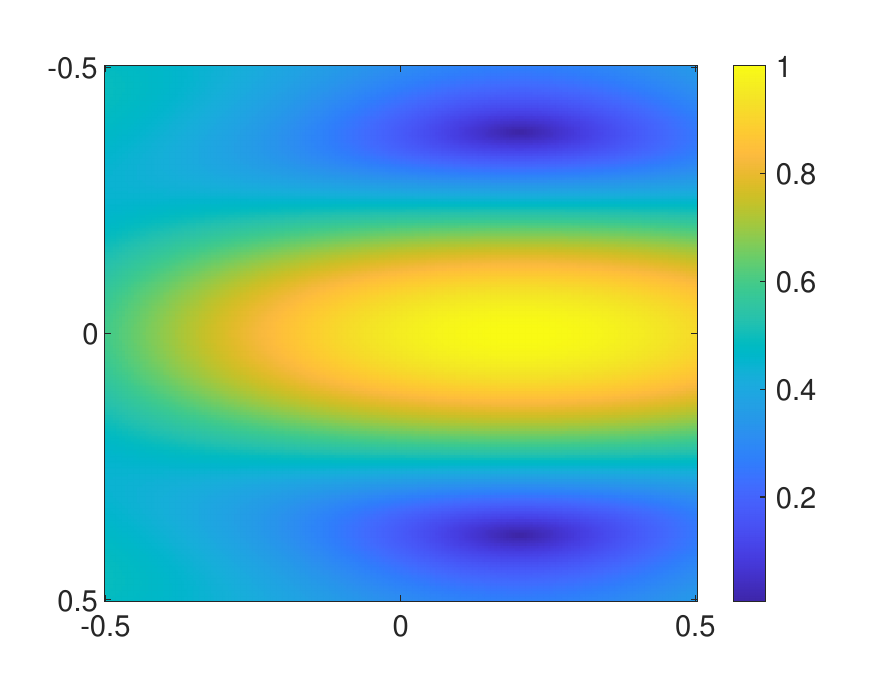}
    \end{subfigure}
    \begin{subfigure}[b]{0.3\textwidth}
        \centering
        \includegraphics[trim={3cm 9cm 3.2cm 9cm},clip,scale=0.38]{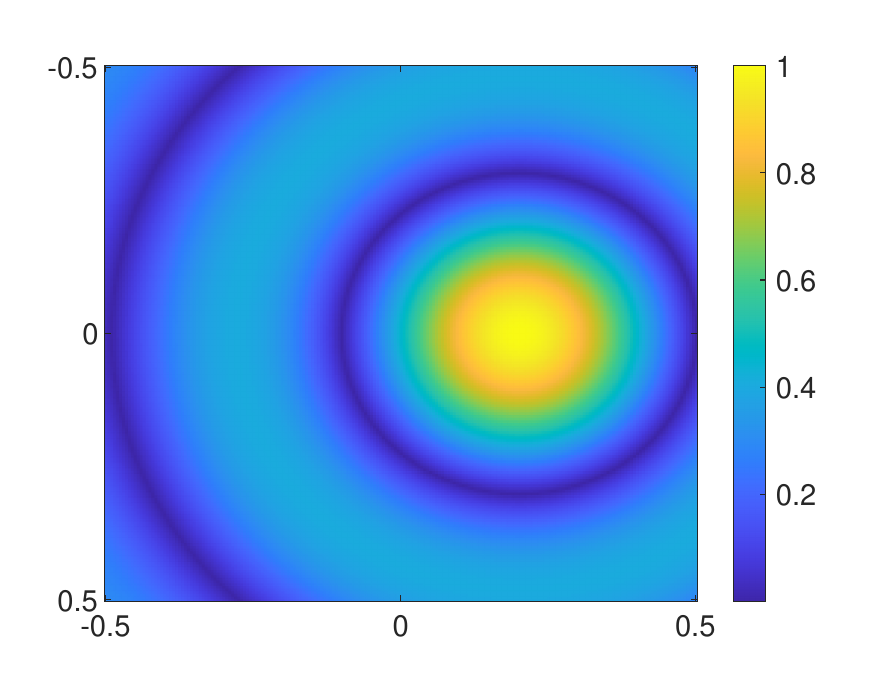}
    \end{subfigure}\\
    \centering
    \begin{subfigure}[b]{0.3\textwidth}
        \centering
        \includegraphics[trim={3cm 9cm 3.2cm 9cm},clip,scale=0.38]{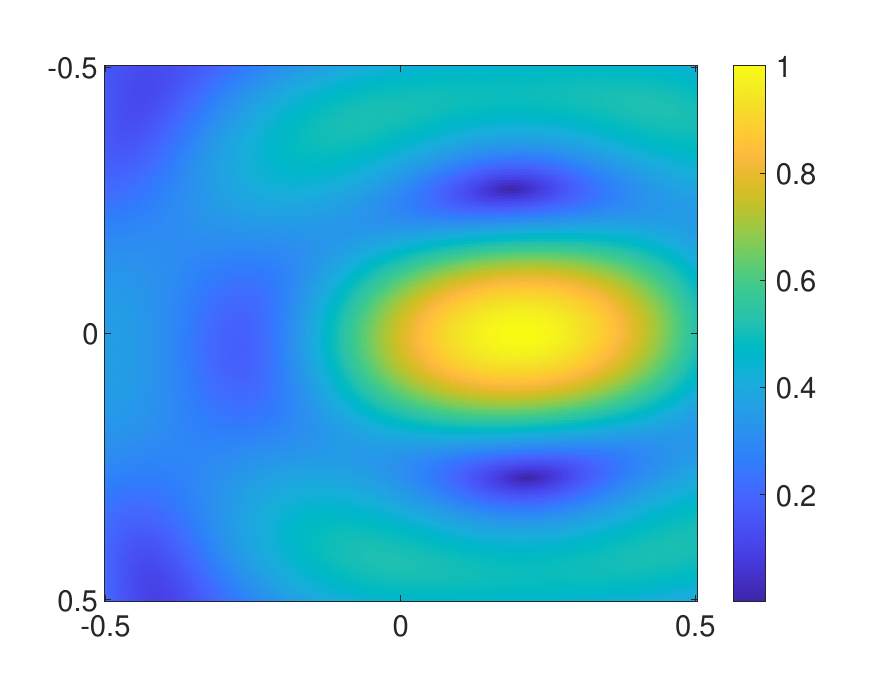}
    \end{subfigure}
    \begin{subfigure}[b]{0.3\textwidth}
        \centering
        \includegraphics[trim={3cm 9cm 3.2cm 9cm},clip,scale=0.38]{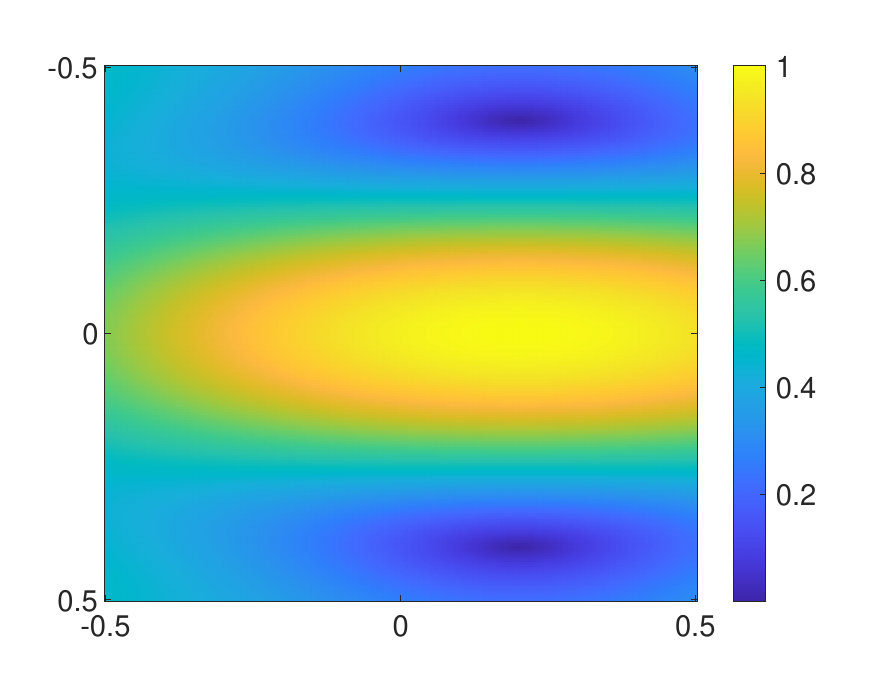}
    \end{subfigure}
    \begin{subfigure}[b]{0.3\textwidth}
        \centering
        \includegraphics[trim={3cm 9cm 3.2cm 9cm},clip,scale=0.38]{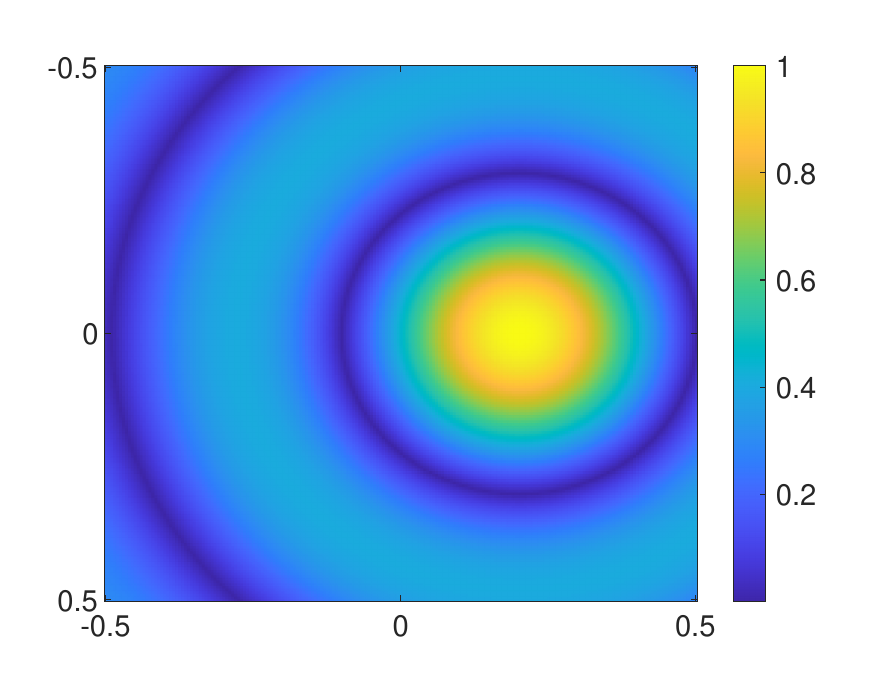}
    \end{subfigure}
    \caption{Comparison of the weighted inner product $\langle \Psi_x, w \Psi_z \rangle_{L^2(\Gamma)}$ (left), the standard boundary inner product $\langle \Psi_x, \Psi_z \rangle_{L^2(\Gamma)}$ (middle), and the full-sphere inner product $\langle \Psi_x, \Psi_z \rangle_{L^2(\bS)}$ (right), for $x \in [-0.5,0.5]^2$ with $k = 8$, $\alpha = \tfrac{\pi}{3}$, and $z = (0.2,0)$. The first row corresponds to $N_y = 12$ measurement points, while the second row corresponds to $N_y = 64$.}
\label{fig_sim}
\end{figure}

From Fig.\,\ref{fig_sim}, it is evident that the weighted inner product $\langle \Psi_x, w \Psi_z \rangle_{L^2(\Gamma)}$ provides a reasonable approximation to the desired full-sphere quantity $\langle \Psi_x, \Psi_z \rangle_{L^2(\bS)}$. Comparing the first column with the two right columns shows that the weighted formulation significantly improves the approximation relative to the unweighted boundary inner product. Moreover, a direct comparison of the first-row and second-row plots demonstrates that increasing the number of measurement points $N_y$ yields a systematic improvement in accuracy: as $N_y$ grows, the weighted inner product converges towards $\langle \Psi_x, \Psi_z \rangle_{L^2(\bS)}$. This confirms the robustness of the weighting strategy with respect to discretization density.

To extend the computation of $w$ to three dimensions, we recall the Jacobi-Anger expansion in \eqref{Jacobi_anger}. For $z \in B(0,1)$, we obtain
\begin{align}
\int_{B(0,1)} e^{ikz\cdot y} \, dz 
&= 4\pi \int_{0}^1 |z|^2 j_0(k|z|) \, d|z| 
   \sum_{m=-n}^n \int_{\bS^2} Y_n^m(\hat{z}) \sin \theta_z \, d\sigma_{z} \, \overline{Y_n^m(y)} = \frac{4\pi}{k^{1/2}} \, j_1(k) \,,
\end{align}
where $j_0$ and $j_1$ denote spherical Bessel functions.

For the aperture $\Gamma$ defined in \eqref{def_Gamma}, we construct the basis functions
\[
v_n^m(y(\theta,\phi)) := l_{n,m} \, e^{i m \pi \phi / \alpha} \, P_n^m\!\Big(\cos\!\Big(\tfrac{\pi}{\beta}\theta\Big)\Big), 
\quad l_{n,m} := \sqrt{\tfrac{2n+1}{4\pi} \cdot \tfrac{(n-m)!}{(n+m)! \, \alpha \beta}} \,,
\]
where $P_n^m$ are associated Legendre polynomials and $v_n^m$ form an orthonormal system in $L^2(\Gamma)$.

Defining $(UP_M): \mathbb{R}^{(2M+1)(M+1)} \to L^2(D)$ by
\begin{equation}
(UP_M \beta)(z) = \sum_{m,n} \beta_{m,n} \int_{\Gamma} e^{ikz\cdot y} \, v_n^m(y) \, d\sigma_y \,,
\end{equation}
its adjoint is given for $f \in L^2(\Omega)$ by
\begin{equation}
((UP_M)^*f)([m,n]) = \int_{\Omega} \int_{\Gamma} e^{-ikz\cdot \ty} \, v_n^m(\ty) \, d\sigma_{\ty} \, f(z) \, dz \,,
\end{equation}
where the index notation $[m,n] = n^2+n+m+1$ provides a bijective enumeration.

Consequently, the matrix operator is
\begin{align}
\label{LHS_w_R3}
(UP_M)^*(UP_M)([m,n], [\tilde{m},\tilde{n}]) 
&= \int_{\Omega} \int_{\Gamma} e^{-ikz\cdot \ty} \, v_{\tilde{n}}^{\tilde{m}}(\ty) \, d\sigma_{\ty}
    \left[ \int_{\Gamma} e^{ikz\cdot y} \, v_n^m(y) \, d\sigma_y \right] dz \\
&= 4\pi \int_{\Gamma} \left[ \int_{\Gamma} \frac{j_1(k|y-\ty|)}{(k|y-\ty|)^{1/2}} \, v_n^m(y) \, d\sigma_y \right] 
        v_{\tilde{n}}^{\tilde{m}}(\ty) \, d\sigma_{\ty} \,. \notag
\end{align}

Finally, the weight $w$ in $\mathbb{R}^3$ can be obtained by solving the linear system defined by \eqref{LHS_w_R3}, with right-hand side given by $(UP_M)^*v$ for $v(z) = j_0(k|z|)$. Importantly, this construction of $w$ is independent of the choice of mesh points in $\Omega$, and thus the same $w$ applies universally to any discrete sampling configuration. Moreover, suppose the kernel function is chosen as $v(z) = J_0(k|z|)$. In that case, the computed $w$ continues to satisfy, at least approximately, the theoretical conclusions established in Section~\ref{sec_finite}, since these rely fundamentally on the approximation $K_w(x,z) \approx J_0(k|x-z|)$.

\subsection{Stability with noisy measurements}
\label{sec_noisy}

We now study stability when the measured discrete far-field matrix is contaminated by additive noise.  Let
\[
F^\delta = F + \varepsilon\,,
\]
where $\varepsilon$ is a noise matrix with $\|\varepsilon\|_2 \le \delta$.  Write the weighted operators
\[
F_w := w^{1/2} F w^{1/2}, \qquad F_w^\delta := w^{1/2} F^\delta w^{1/2}
= F_w + W,\qquad W := w^{1/2}\,\varepsilon\,w^{1/2}\,.
\]
The exact and noisy solutions of the weighted far-field equations satisfy
\[
F_w g_x^w = \eta_x, \qquad F_w^\delta g_x^{w,\delta} = \eta_x\,,
\]
with $\eta_x = w^{1/2}\Psi_x$ as before. Subtracting these two equations gives
\[
F_w\bigl(g_x^{w,\delta}-g_x^w\bigr) = -W g_x^{w,\delta}.
\]
Assuming $F_w$ is invertible and $\|F_w^{-1}\|_2\|W\|_2<1$ which implies $F_w^\delta$ is also invertible, we obtain the exact identity
\[
g_x^{w,\delta}-g_x^w = -F_w^{-1} W g_x^{w,\delta}
= -F_w^{-1} W (F_w+W)^{-1}\eta_x\,.
\]
Hence, the difference is bounded by
\begin{equation}
\label{noisy_bound}
\|g_x^{w,\delta}-g_x^w\|_2
\le \frac{\|F_w^{-1}\|_2^2 \,\|W\|_2}{1-\|F_w^{-1}\|_2\|W\|_2}\,\|\eta_x\|_2\,.
\end{equation}
Using $\|W\|_2 \le \|w\|_\infty \,\|\varepsilon\|_2 \le \|w\|_\infty \,\delta$ and $\|\eta_x\|_2=\|w^{1/2}\Psi_x\|_2$, \eqref{noisy_bound} yields an explicit stability estimate in terms of the noise level $\delta$, the weight $w$, and the conditioning of $F_w$.

Thus the sensitivity to noise is governed by the factor $\|F_w^{-1}\|_2$ and $\|w\|_\infty$, while $\|\varepsilon\|_2$ measures the noise magnitude.  For the classical linear sampling method, one has $w\equiv 1$, and \eqref{noisy_bound} recovers the usual error dependence.

Next, we relate $\|F_w^{-1}\|_2$ to the spectrum of the discrete kernel.  Under the Born approximation, $K_w(\Omega_{N_z})=U^* w U$ is the kernel matrix on the sampling set $\Omega_{N_z}$ and, by construction, $F_w$ and $K_w(\Omega_{N_z})$ have the same nonzero singular values. Hence
\[
\|F_w^{-1}\|_2 = \frac{1}{\min_j \sigma_j\!\bigl(K_w(\Omega_{N_z})\bigr)}\,,
\]
where $\sigma_j(\cdot)$ denotes the nonzero singular values.

To explain why the weight chosen by \eqref{normal_equation} improves stability, recall that $w$ is selected to match the full-aperture kernel:
\[
\arg\min_{w\in\Lambda_M}\|K_w(0,\cdot)-v(\cdot)\|^2_{L^2(\Omega)}\,,
\]
and, discretely, to make the matrix $K_w(\Omega_{N_z})$ close to the reference matrix $v(\Omega_{N_z})$ whose entries are $J_0(k|z_m-z_n|)$. Recall the Hoffman–Wielandt theorem
\[
\sum_j\bigl(\sigma_j(K_w(\Omega_{N_z}))-\sigma_j(v(\Omega_{N_z}))\bigr)^2
\le \big\|K_w(\Omega_{N_z})-v(\Omega_{N_z})\big\|_F^2\,.
\]
Thus small Frobenius-norm discrepancy implies the singular values of $K_w(\Omega_{N_z})$ are close to those of $v(\Omega_{N_z})$. The full-aperture matrix $v(\Omega_{N_z})$ typically has a significantly larger minimal singular value than the limited-aperture matrix $F$; therefore, a weight $w$ that enforces $K_w\approx v$ will increase the smallest singular value and hence reduce $\|F_w^{-1}\|_2$.


Finally, we validate these observations numerically. For $D=B(0,1/2)$, based on expression \eqref{noisy_bound}, we compute and plot the quantities
\[
\log_{10}\|F^{-1}\|_2 \qquad\text{and}\qquad
\log_{10}\bigl(\|F_w^{-1}\|_2\|w^{1/2}\|_2\bigr),
\]
for various apertures $\alpha$ and numbers of measurement points $N_y$. The results in Fig.~\ref{fig_stable} show that $\|F_w^{-1}\|_2\|w^{1/2}\|_2$ is consistently much smaller than $\|F^{-1}\|_2$, indicating that the weighted inner product significantly improves stability compared with the classical unweighted method.

\begin{figure}
\captionsetup{width=.9\linewidth} \centering \begin{subfigure}[b]{0.48\textwidth} \centering \includegraphics[trim={2.8cm 8.4cm 3.2cm 8.8cm}, clip, scale=0.52]{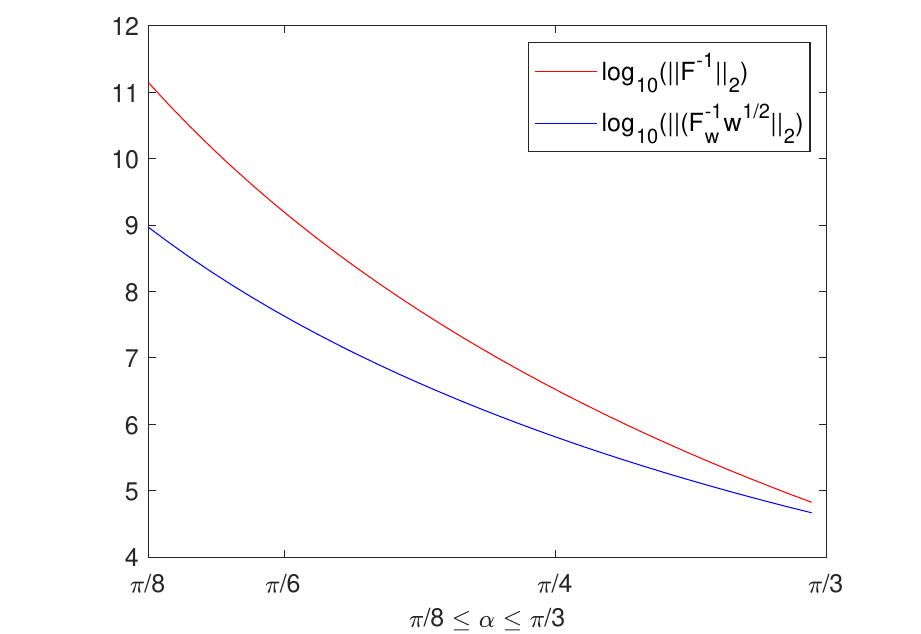} \label{fig_stable_aperture} \end{subfigure} \begin{subfigure}[b]{0.48\textwidth} \centering \includegraphics[trim={2.8cm 8.4cm 3.2cm 8.8cm}, clip, scale=0.52]{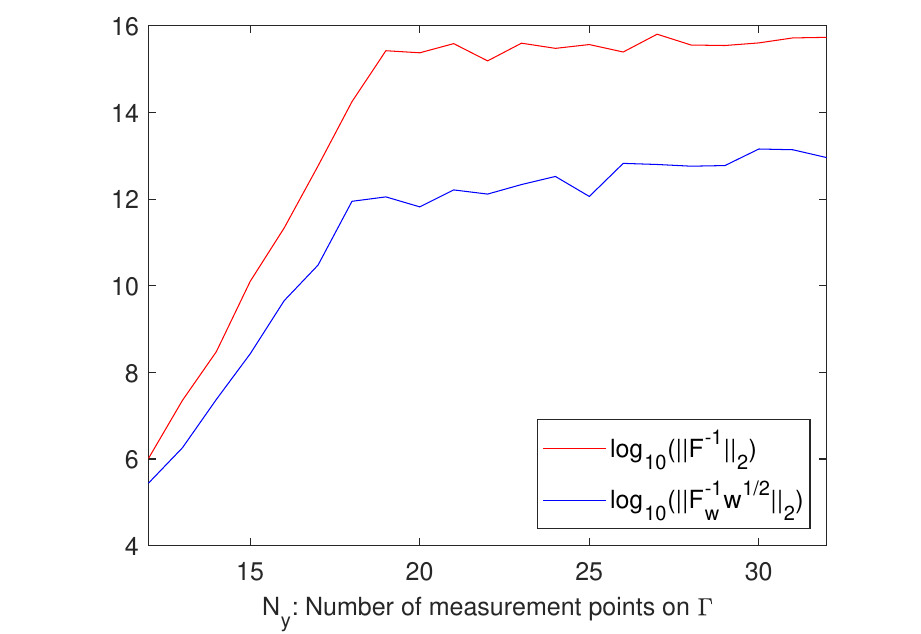} \end{subfigure} \caption{Plots of $\log_{10}(\| F^{-1} \|_2)$ and $\log_{10}(\| F_w^{-1} w^{1/2} \|_2)$ for the linear sampling method and the proposed weighted method. Left: varying aperture $\pi/8 \leq \alpha \leq \pi/3$ with $N_y = 12$, $k = 4$. Right: varying numbers of measurement points $12 \leq N_y \leq 36$ with $\alpha = \pi/3$, $k = 6$.} 
\label{fig_stable}
\end{figure}

\section{Numerical experiments}
\label{sec_numerics}

In this section, we evaluate the accuracy and stability of the proposed novel sampling method and compare it with several existing methods, including the factorization method \cite{factorization_book}, the linear sampling method (LSM) \cite{book_cakoni}, and the multiple signal classification (MUSIC) method \cite{music_first_2003}. In all examples, we set $q(z)=1$ and take $D \subset B(0,1)$ as the support of the inhomogeneous medium, using various shapes to validate the theoretical results in Section~\ref{sec_finite}.
Measurement data is generated by solving \eqref{int_equ_u} to obtain $u^{\infty}(y, \bod)$ for $y, \bod \in \Gamma_{N_y}$, where $\Gamma_{N_y} = \{ y_j \in \Gamma \mid j = 1, \ldots, N_y \}$. The noisy measurement data is then given by
\begin{equation}
    \tilde{u}^{\infty}(y, \bod) = u^{\infty}(y, \bod) + E(u^{\infty}) \delta \epsilon \,,
\end{equation}
where $E(u^{\infty})$ denotes the mean value of $u^{\infty}$ on $\Gamma_{N_y} \times \Gamma_{N_y}$, $\epsilon$ is a random variable uniformly distributed in $[-1,1]$, and $\delta$ represents the noise level. We assume $\tilde{u}^{\infty}$ represents the noisy data obtained from real experiments.

For the index function computation, we obtain $w$ by solving \eqref{normal_equation} proposed in Section~\ref{sec_construction_w} during the offline phase. To further enhance the numerical stability in solving the far-field equation or \eqref{def_gzw} under noisy data, Tikhonov regularization combined with Morozov's discrepancy principle is employed for the factorization method, the linear sampling method, and the proposed method, in the sense that we solve
\[
g_x^{w,\tau} = \arg\min_{g_x^w}\left\{\|F_w g_x^{w}- \eta_x\|_2^{2}+\tau\|g_x^{w}\|_2^{2}\right\},
\]
where the regularization parameter $\tau>0$ is chosen such that
\[
\|F_w g_x^{w,\tau}- \eta_x\|_2 = \delta\,,
\]
with $\delta$ an estimate of the noise level.

In the plots that follow, $I_w$, $I_{LSM}$, $I_{fac}$, and $I_{MUS}$ denote reconstructions using the proposed method in \eqref{def_index}, the LSM in \eqref{def_index_LSM}, the factorization method in \cite{factorization_first}, and the MUSIC method from \cite{music_first_2003}, respectively. The last plot in each figure shows the exact location of the inhomogeneous medium. To improve visualization, we plot the square root of these index functions to identify the boundary of each inclusion better.

\subsection[NE in R2]{Numerical experiments in $\mathbb{R}^2$}

We first present numerical examples in $\mathbb{R}^2$. In the last plot of each figure in this subsection, red dots on $\bS$ represent the locations of measurement points, and the thin black line denotes $\bS$.

\textbf{Example 1.} In this example, shown in Fig.\,\ref{fig_ex1}, we compare the reconstruction results of the proposed method, MUSIC, and the factorization method under various apertures. The inhomogeneous medium is two rectangular bars located at $[0.2,0.4] \times [-0.25,0.25]$ and $[-0.4,0.2] \times [-0.25,0.25]$. The noise level is consistently set at 5\%, and the aperture values are $\pi/2$, $\pi/3$, and $\pi/6$, with 16 measurement points. Fig.\,\ref{fig_ex1} displays the reconstructions from left to right for the proposed method, MUSIC, and the factorization method. The plots in the first to third rows correspond to apertures of $\pi/2$, $\pi/3$, and $\pi/6$, respectively.

\begin{figure}[H]
   \centering
     \begin{subfigure}[b]{1\textwidth}
    \includegraphics[trim={4.5cm 8cm 4.5cm 8cm},clip,scale = 0.33]{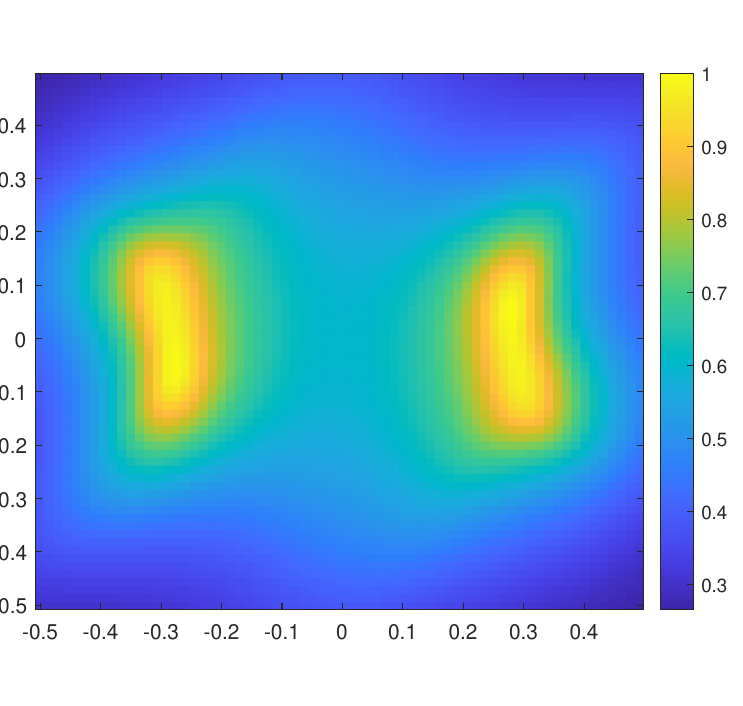}
    \includegraphics[trim={4.5cm 8cm 4.5cm 8cm},clip,scale = 0.33]{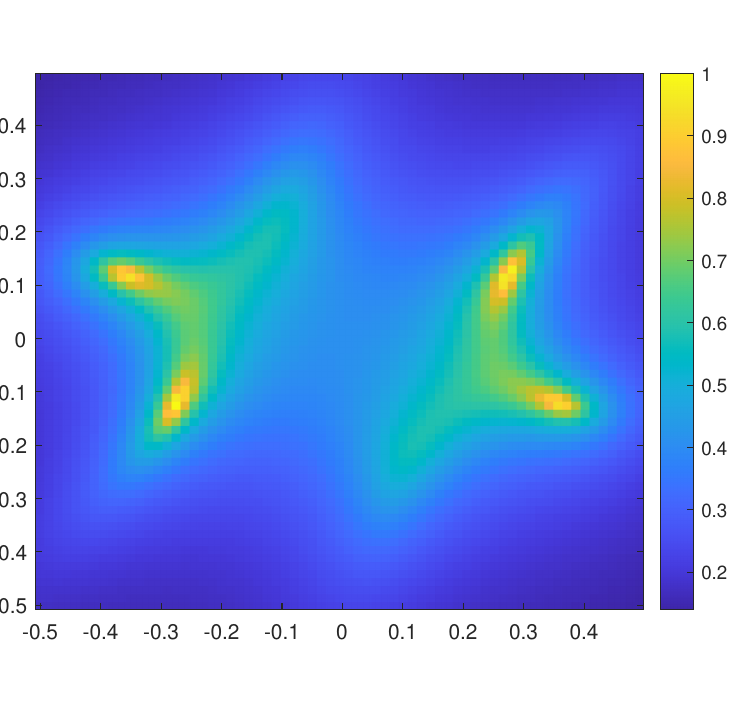}
    \includegraphics[trim={4.5cm 8cm 4.5cm 8cm},clip,scale = 0.33]{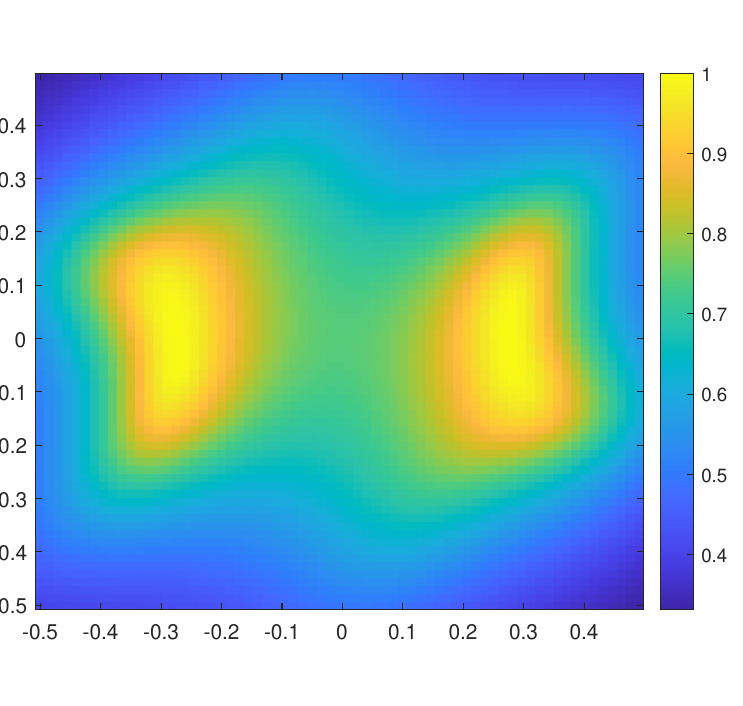}
    \includegraphics[trim={4.5cm 8cm 4.5cm 8cm},clip,scale = 0.33]{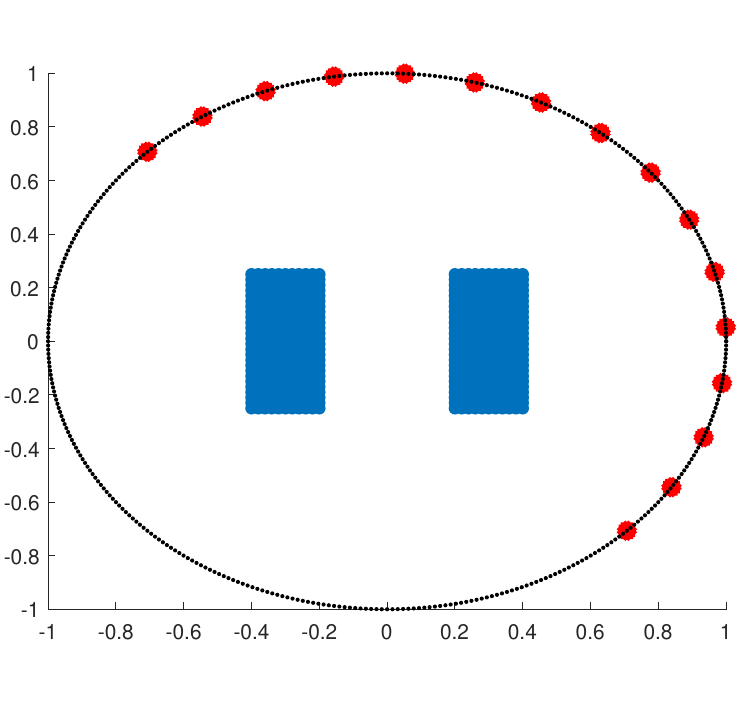}
    \subcaption{$I_w$, $I_{MUS}$, $I_{fac}$, exact inclusion (from left to right) with aperture $\pi/2$} 
    \end{subfigure}
   \centering
     \begin{subfigure}[b]{1\textwidth}
    \includegraphics[trim={4.5cm 8cm 4.5cm 8cm},clip,scale = 0.33]{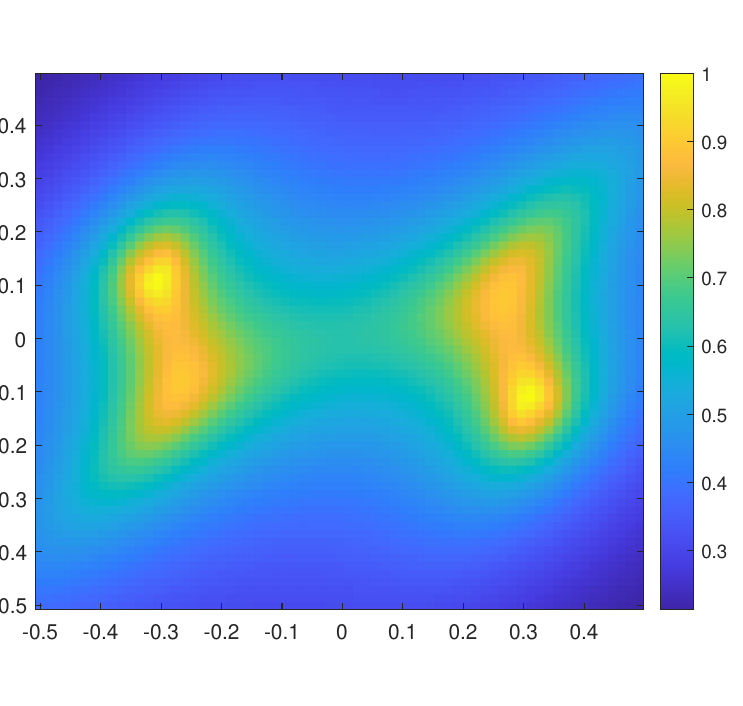}
    \includegraphics[trim={4.5cm 8cm 4.5cm 8cm},clip,scale = 0.33]{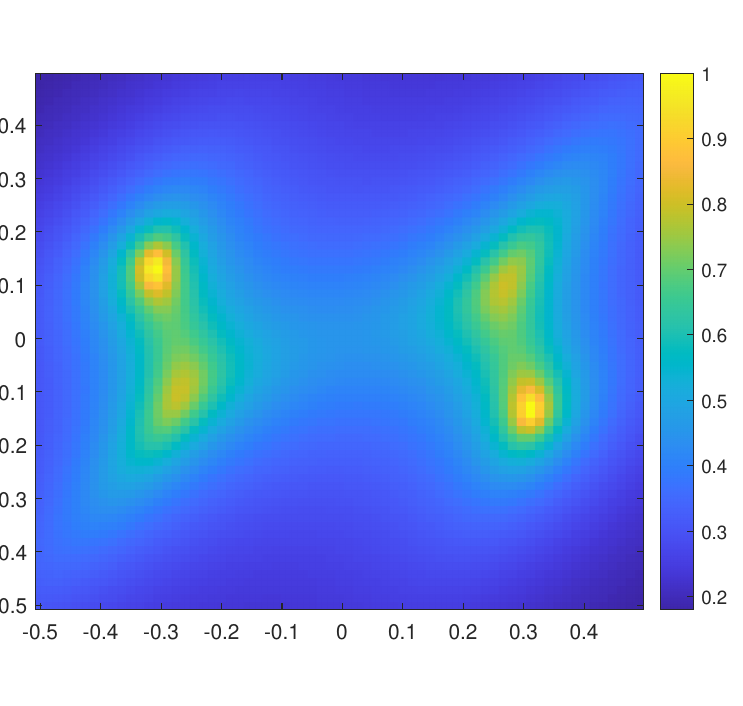}
    \includegraphics[trim={4.5cm 8cm 4.5cm 8cm},clip,scale = 0.33]{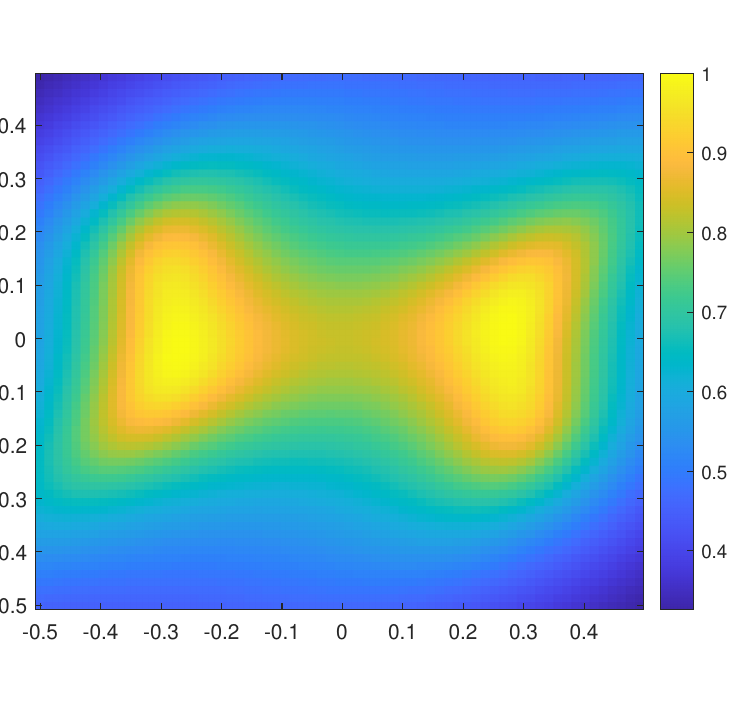}
    \includegraphics[trim={4.5cm 8cm 4.5cm 8cm},clip,scale = 0.33]{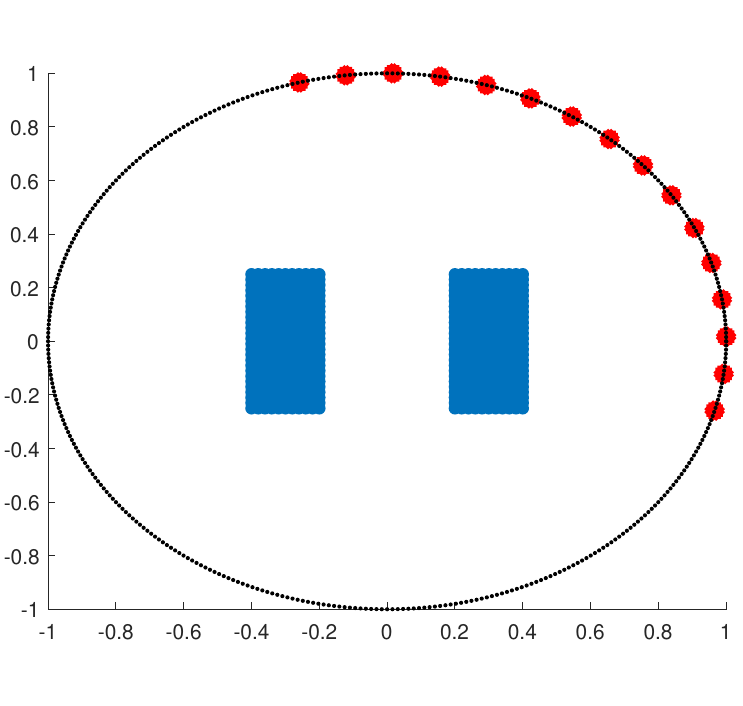}
    \subcaption{$I_w$, $I_{MUS}$, $I_{fac}$, exact inclusion (from left to right) with aperture $\pi/3$} 
     \end{subfigure}
   \centering
     \begin{subfigure}[b]{1\textwidth}
    \includegraphics[trim={4.5cm 8cm 4.5cm 8cm},clip,scale = 0.33]{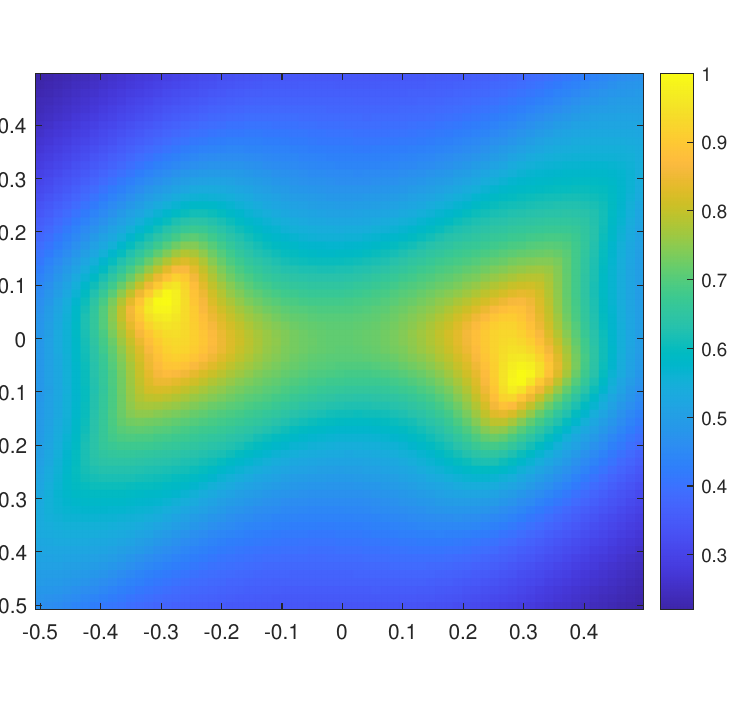}
    \includegraphics[trim={4.5cm 8cm 4.5cm 8cm},clip,scale = 0.33]{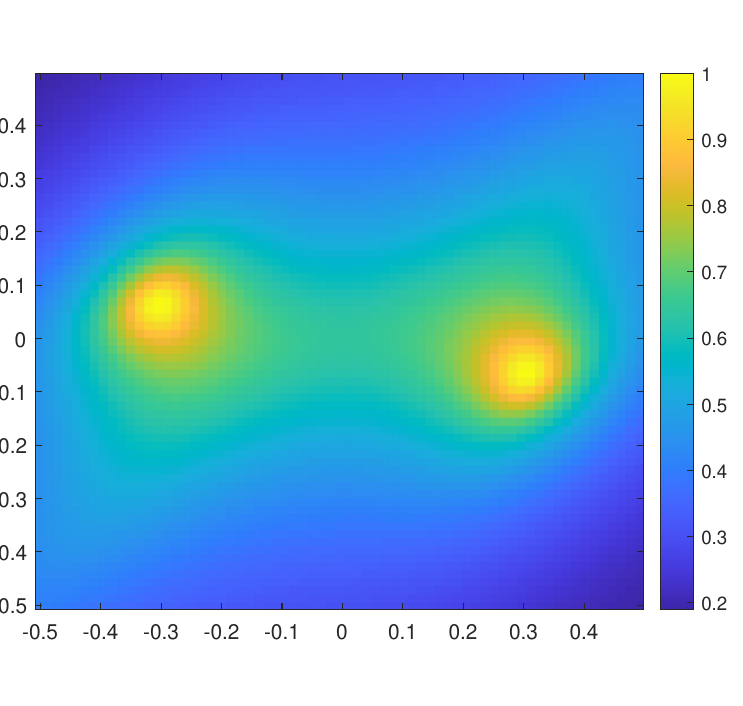}
    \includegraphics[trim={4.5cm 8cm 4.5cm 8cm},clip,scale = 0.33]{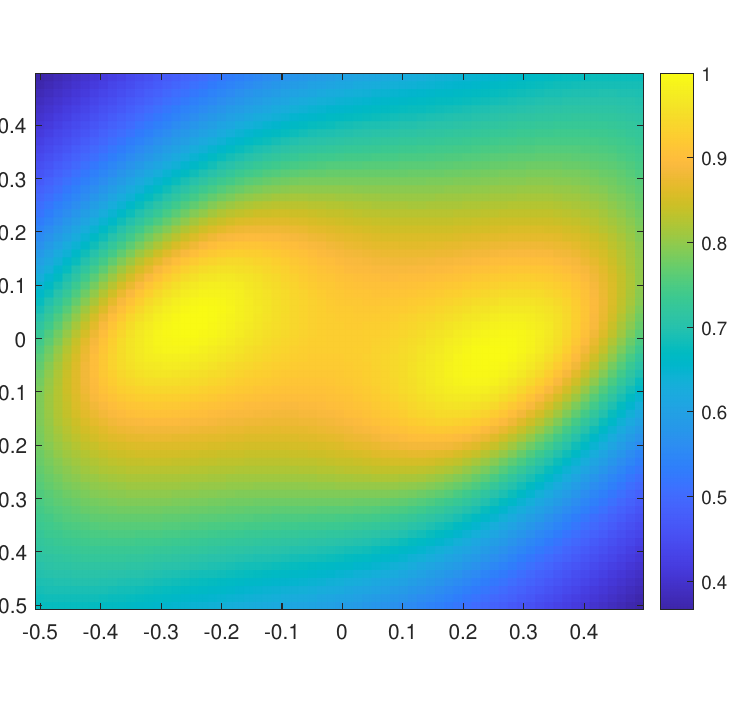}
    \includegraphics[trim={4.5cm 8cm 4.5cm 8cm},clip,scale = 0.33]{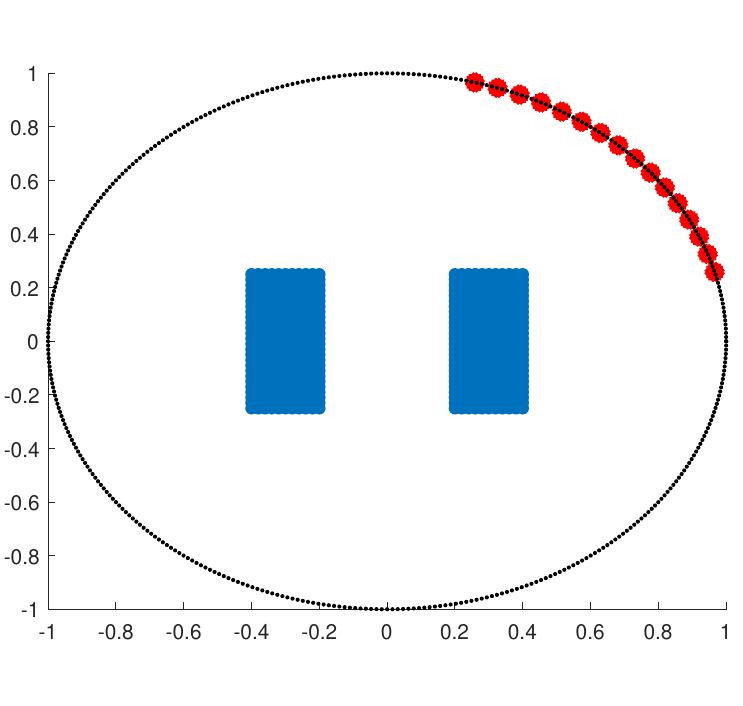}
    \subcaption{$I_w$, $I_{MUS}$, $I_{fac}$, exact inclusion (from left to right) with aperture $\pi/6$} 
    \end{subfigure}
    \caption{\textbf{Example 1.} Reconstructions using $I_w$, $I_{MUS}$, and $I_{fac}$ (from left to right) for different apertures, $k = 6$, 5\% random noise, and 16 measurement points. The final column shows the exact inhomogeneous medium and the locations of measurement points.}
\label{fig_ex1}
\end{figure}

From Fig.\,\ref{fig_ex1}, it is evident that as the aperture decreases, the accuracy of the reconstruction deteriorates. The proposed method, depicted in the first column of each row, remains robust even with limited aperture data, effectively locating the inclusion when the aperture is $\pi/3$ and $\pi/6$. In contrast, the reconstructions by MUSIC and the factorization method, shown in the last two columns, are less accurate and fail to identify the two inhomogeneous inclusions under the same conditions clearly.

\textbf{Example 2.} In this example, illustrated in Fig.\,\ref{fig_ex2}, we reconstruct four point scatters located at $[\pm 0.3, \pm 0.3]$. The noise level is set at 5\%, with an aperture of $\pi/3$ and 16 measurement points on the boundary.

\begin{figure}[H]
    \centering
     \begin{subfigure}[b]{0.3\textwidth}
    \centering
    \includegraphics[trim={4.5cm 8cm 4.5cm 8cm},clip,scale = 0.33]{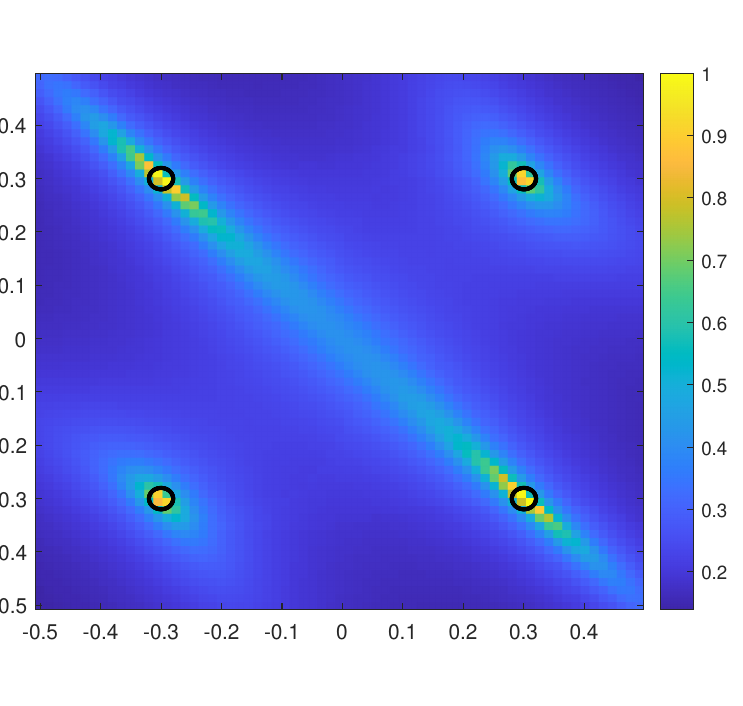}
    \caption{$I_w$}
    \end{subfigure}
    \quad     
    \begin{subfigure}[b]{0.3\textwidth}
    \centering
    \includegraphics[trim={4.5cm 8cm 4.5cm 8cm},clip,scale = 0.33]{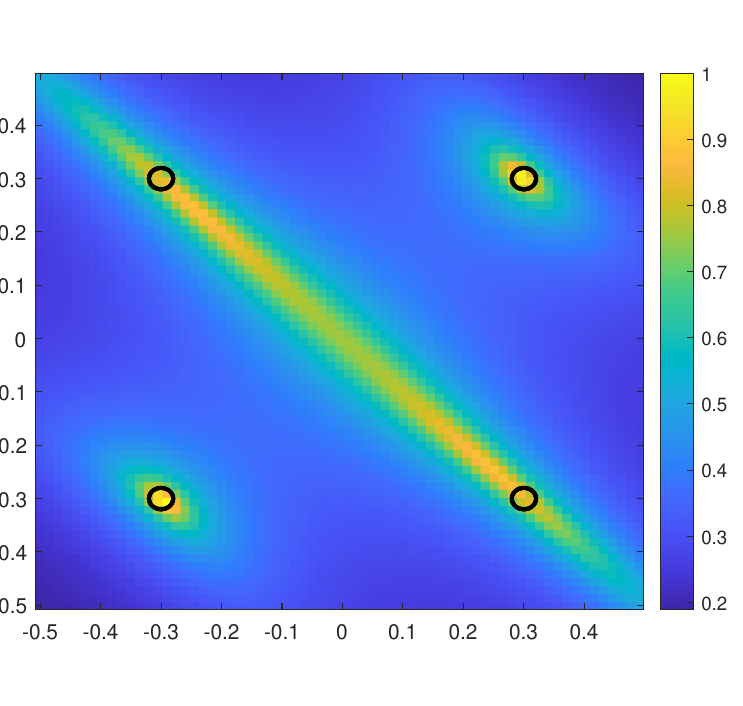}
    \caption{$I_{LSM}$}
    \end{subfigure}
    \quad
    \begin{subfigure}[b]{0.3\textwidth}
    \centering
    \includegraphics[trim={4.5cm 8cm 4.5cm 8cm},clip,scale = 0.33]{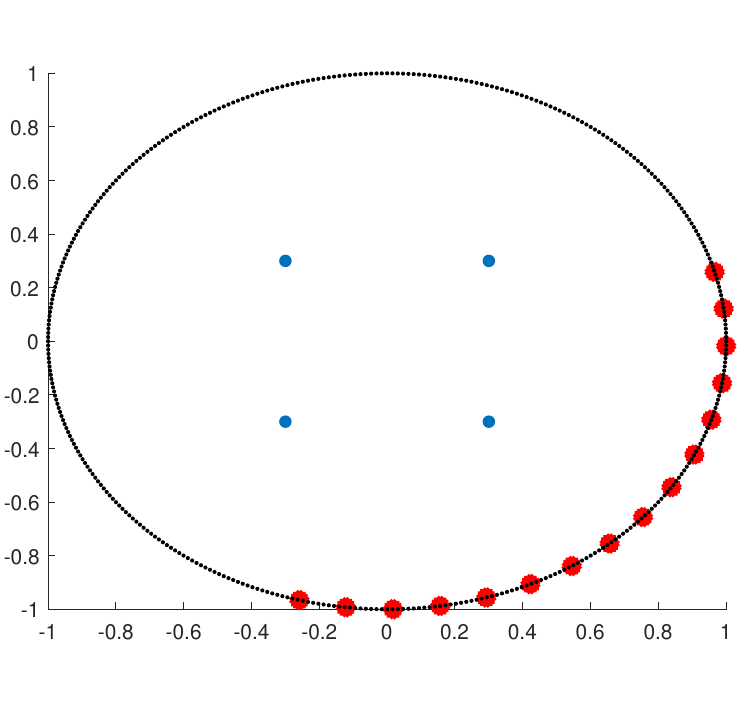}
    \caption{Exact inclusion}
    \end{subfigure}
    \caption{\textbf{Example 2.} Reconstructions of $I_w$ (left) and $I_{LSM}$ (middle) with aperture $\pi/3$, $k = 6$, 5\% random noise, and 16 measurement points; and the exact inclusion (right).}
\label{fig_ex2}
 \end{figure}

From Fig.\,\ref{fig_ex2}, it is evident that the proposed method, depicted in the left plot, clearly recovers the locations of the four point scatters. In contrast, the $I_{LSM}$ method, shown in the middle plot, has difficulty identifying the locations of the point scatters, particularly those at $[-0.3, 0.3]$ and $[0.3, -0.3]$.

\textbf{Example 3.} In this example, we reconstruct a non-convex kite-shaped inclusion, with its boundary parameterized as $0.09 \times [0.15 \sin(t), \, (\cos(t) + 0.65 \cos(2t) - 0.65)]$ for $t \in [0, 2\pi]$, with $k = 12$. The noise level is set at 5\%, and the aperture is $\pi/3$ with 32 measurement points on $\bS$. The black line in the first two plots indicates the boundary of the inclusion. The results are shown in Fig.\,\ref{fig_ex3}.

\begin{figure}[H]
    \centering
     \begin{subfigure}[b]{0.3\textwidth}
    \centering
    \includegraphics[trim={4.5cm 8cm 4.5cm 8cm},clip,scale = 0.33]{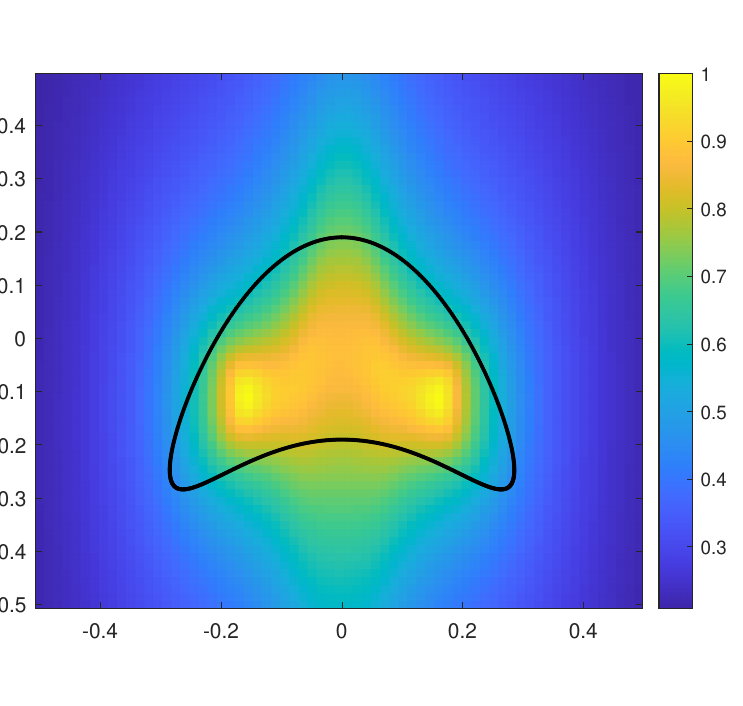}
    \caption{$I_w$}
    \end{subfigure}
    \quad     
    \begin{subfigure}[b]{0.3\textwidth}
    \centering
    \includegraphics[trim={4.5cm 8cm 4.5cm 8cm},clip,scale = 0.33]{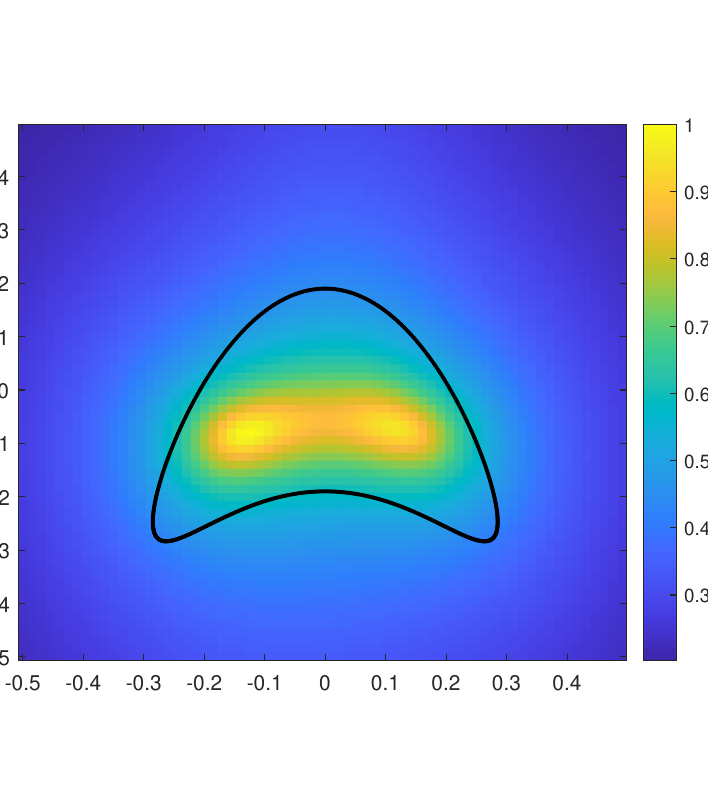}
    \caption{$I_{LSM}$}
    \end{subfigure}
    \quad
    \begin{subfigure}[b]{0.3\textwidth}
    \centering
    \includegraphics[trim={4.5cm 8cm 4.5cm 8cm},clip,scale = 0.33]{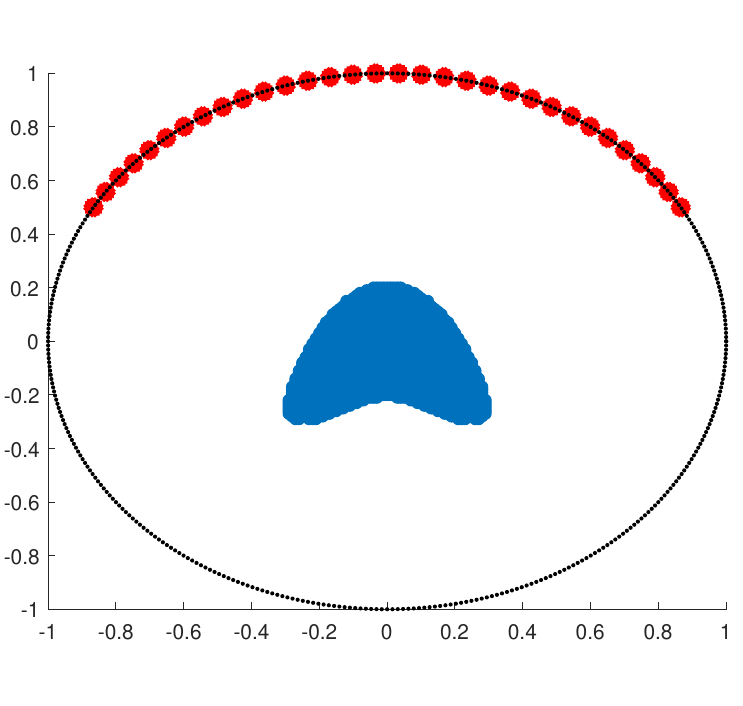}
    \caption{Exact inclusion}
    \end{subfigure}
   \caption{\textbf{Example 3.} Reconstructions of $I_w$ (left) and $I_{LSM}$ (middle) with aperture $\pi/3$, $k = 12$, 5\% random noise, and 32 measurement points; and the exact inclusion (right).}
\label{fig_ex3}
 \end{figure}

From Fig.\,\ref{fig_ex3}, the reconstruction using the proposed method (left plot) accurately identifies the boundary of the kite-shaped inclusion and confirms the presence of a single inclusion. In contrast, the reconstruction using $I_{LSM}$ (middle plot) incorrectly suggests the presence of two separate inclusions.

\textbf{Example 4.} This example, shown in Fig.\,\ref{fig_ex4}, evaluates reconstruction with a smaller wave number, $k = 3$. The inhomogeneous medium is supported in $[0.2, 0.4] \times [-0.25, 0.25]$ and $[-0.4, 0.2] \times [-0.25, 0.25]$. The noise level is 5\%, and the aperture is $\pi/3$ with 32 measurement points on the boundary.

\begin{figure}[H]
    \centering
     \begin{subfigure}[b]{0.3\textwidth}
    \centering
    \includegraphics[trim={4.5cm 8cm 4.5cm 8cm},clip,scale = 0.33]{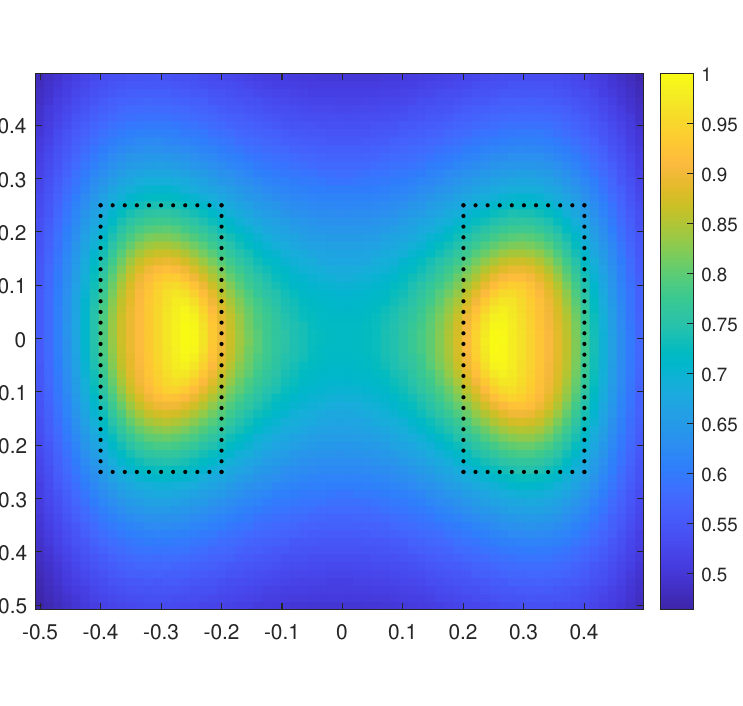}
    \caption{$I_w$}
    \end{subfigure}
    \quad     
    \begin{subfigure}[b]{0.3\textwidth}
    \centering
    \includegraphics[trim={4.5cm 8cm 4.5cm 8cm},clip,scale = 0.33]{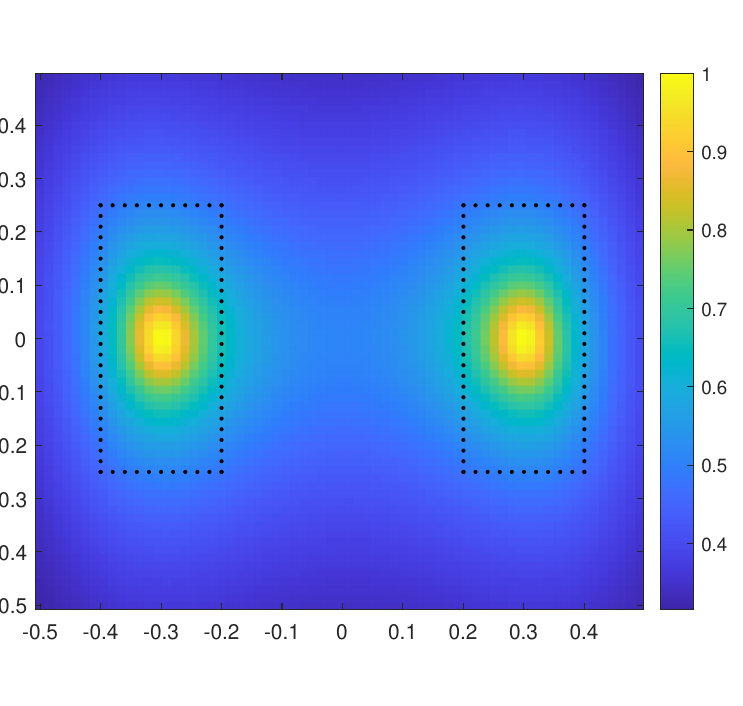}
    \caption{$I_{LSM}$}
    \end{subfigure}
    \quad
    \begin{subfigure}[b]{0.3\textwidth}
    \centering
    \includegraphics[trim={4.5cm 8cm 4.5cm 8cm},clip,scale = 0.33]{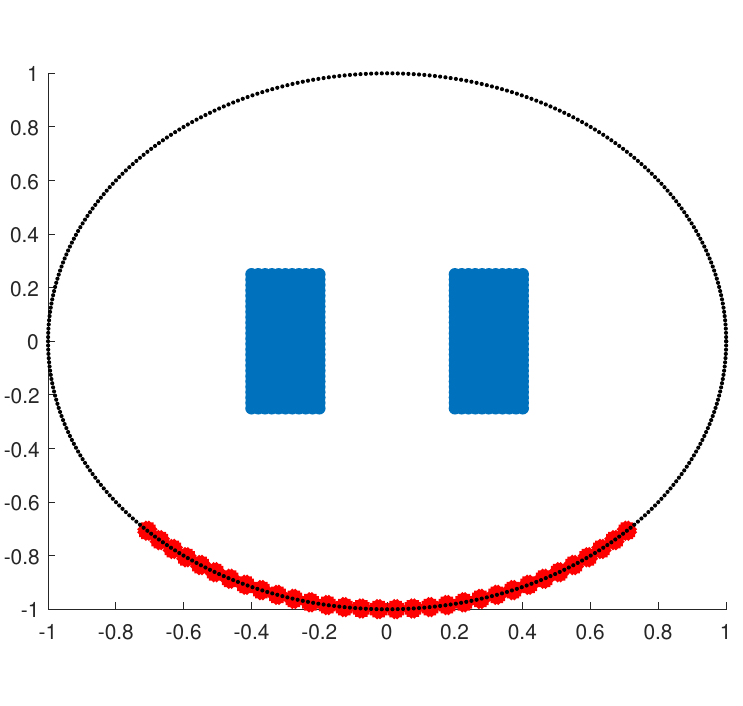}
    \caption{Exact inclusion}
    \end{subfigure}
 \caption{\textbf{Example 4.} Reconstructions of $I_w$ (left) and $I_{LSM}$ (middle) with aperture $\pi/3$, $k = 3$, 8\% random noise, and 32 measurement points; and the exact inclusion (right).}
\label{fig_ex4}
 \end{figure}

From Fig.\,\ref{fig_ex4}, it is clear that the proposed method (left plot) accurately recovers the size and location of the two inclusions, even with a smaller wave number $k = 3$. The reconstruction by $I_{LSM}$ (middle plot) is less accurate compared to the proposed method.

\textbf{Example 5.} In this example, shown in Fig.\,\ref{fig_ex5}, we compare the proposed method with the LSM method and assess the robustness of reconstruction under varying noise levels. We assume two elliptical inhomogeneous inclusions with boundaries parameterized by $[3 \cos t / r^2 - 0.25, \sin t / r^2 - 0.3]$ and $[\cos t / r^2 + 0.3, 3 \sin t / r^2 + 0.25]$ for $t \in [0, 2\pi]$ with $r = 0.08$. The figures present reconstructions with noise levels of 1\%, 10\%, and 20\%. The first row displays results using $I_w$, while the second row shows results using $I_{LSM}$, both with $k = 8$, an aperture of $\pi/3$, and 16 measurement points on the boundary.

\begin{figure}[H]
    \centering
     \begin{subfigure}[b]{1\textwidth}
    \centering
    \includegraphics[trim={4.5cm 9.5cm 4.5cm 9.5cm},clip,scale = 0.33]{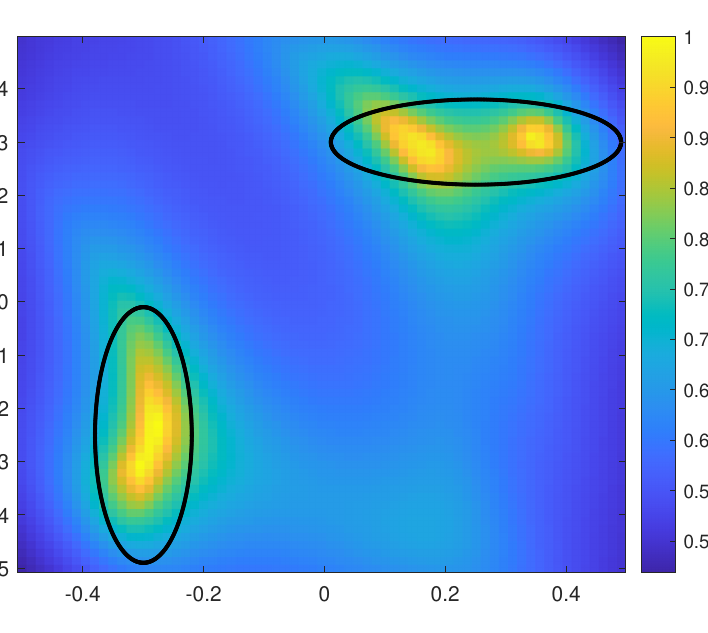}\quad
    \includegraphics[trim={4.5cm 9.5cm 4.5cm 9.5cm},clip,scale = 0.33]{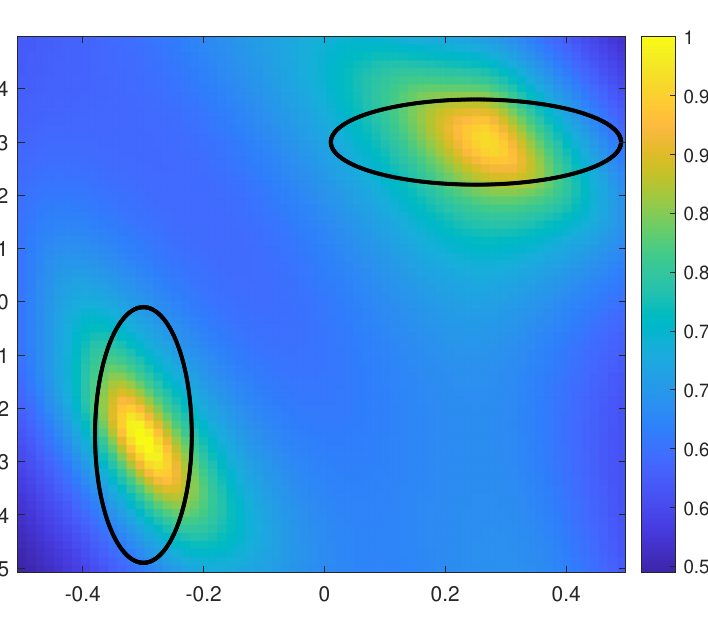}\quad
    \includegraphics[trim={4.5cm 9.5cm 4.5cm 9.5cm},clip,scale = 0.33]{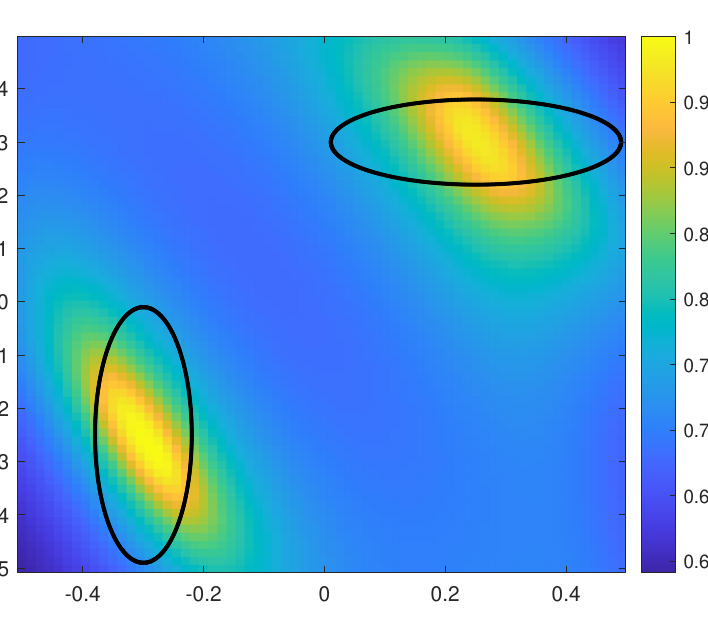}
    \caption{$I_w$, with noise levels of 1\%, 10\%, and 20\% (from left to right).}
    \end{subfigure}
    \begin{subfigure}[b]{1\textwidth}
    \centering
    \includegraphics[trim={4.5cm 9.5cm 4.5cm 9.5cm},clip,scale = 0.33]{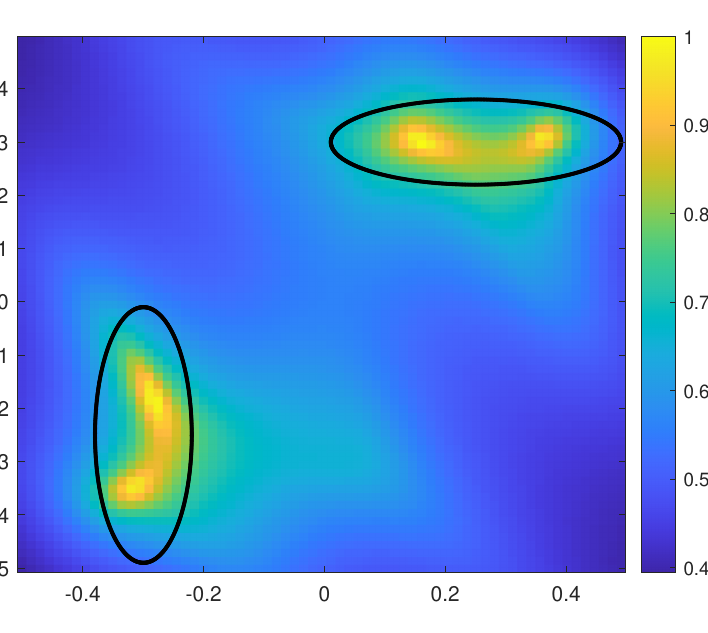}\quad
    \includegraphics[trim={4.5cm 9.5cm 4.5cm 9.5cm},clip,scale = 0.33]{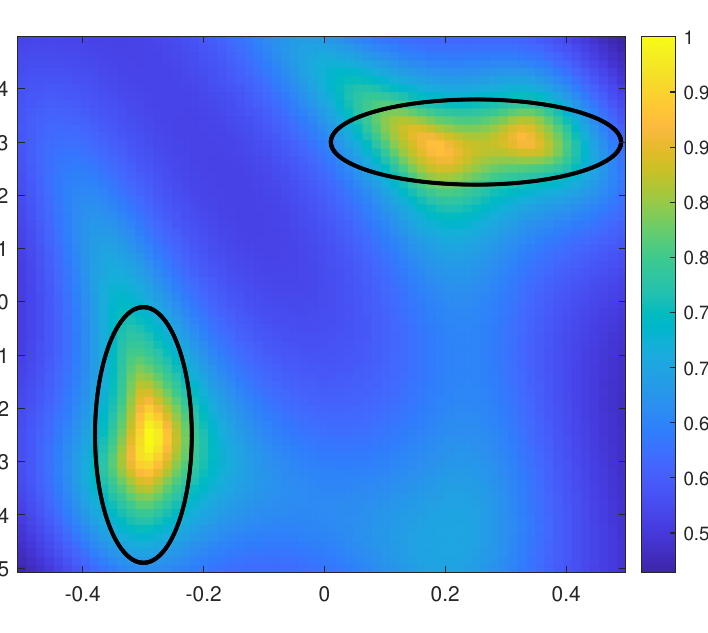}\quad
    \includegraphics[trim={4.5cm 9.5cm 4.5cm 9.5cm},clip,scale = 0.33]{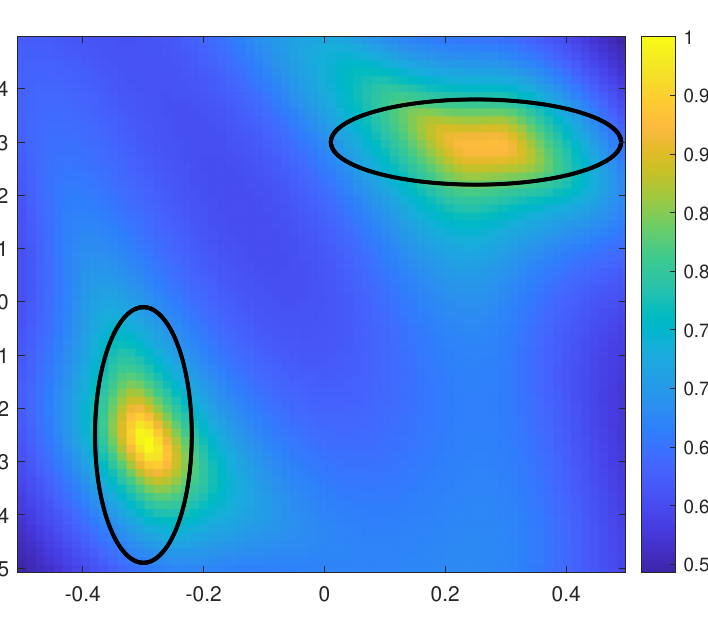}
    \caption{$I_{LSM}$, with noise levels of 1\%, 10\%, and 20\% (from left to right).}
    \end{subfigure}
    \caption{\textbf{Example 5.} Reconstructions at different noise levels with an aperture of $\pi/4$, $k = 8$, and 16 measurement points.}
\label{fig_ex5}
 \end{figure}

From Fig.\,\ref{fig_ex5}, it is observed that at a 1\% noise level (first column), both $I_w$ and $I_{LSM}$ accurately recover the locations of the inclusions. However, as the noise level increases (second and third columns), $I_w$ is notably better at recovering the support of the inclusions, especially the ellipse located at the upper right corner.

\subsection[NE in R3]{Numerical experiments in $\mathbb{R}^3$}

Next, we conduct two numerical experiments in $\mathbb{R}^3$ to validate the proposed method. We use $k = 8$ and a noise level of 10\%. The measurement aperture is $\theta \in [0, \pi/4]$ and $\phi \in [-\pi, \pi]$ in spherical coordinates, with 78 measurement points. The locations of these points are shown by blue dots in the final plot of the figures. To illustrate the three-dimensional reconstruction, we use an isosurface with a truncation value of 0.1 to represent the boundary of the inhomogeneous medium.

\textbf{Example 6.} This example, presented in Fig.\,\ref{fig_ex3d_1}, involves reconstructing two balls centered at $[-0.3, 0.2, 0]$ and $[0.3, 0, 0]$ with radii 0.2 and 0.1, respectively.

\begin{figure}[H]
    \centering
     \begin{subfigure}[b]{0.3\textwidth}
    \centering
    \includegraphics[trim={4.5cm 8cm 4.5cm 8cm},clip,scale = 0.3]{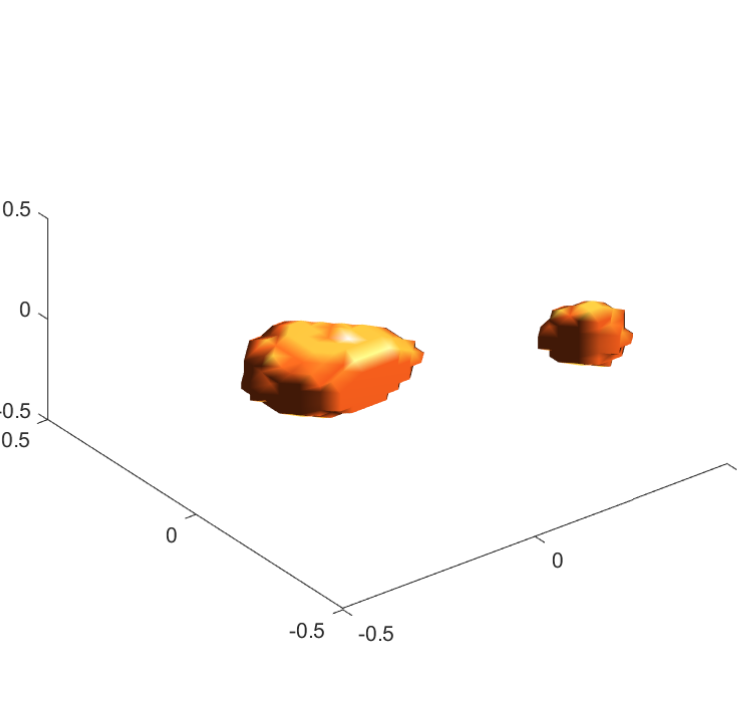}
    \caption{$I_w$}
    \end{subfigure}
    \quad     
    \begin{subfigure}[b]{0.3\textwidth}
    \centering
    \includegraphics[trim={4.5cm 8cm 4.5cm 8cm},clip,scale = 0.3]{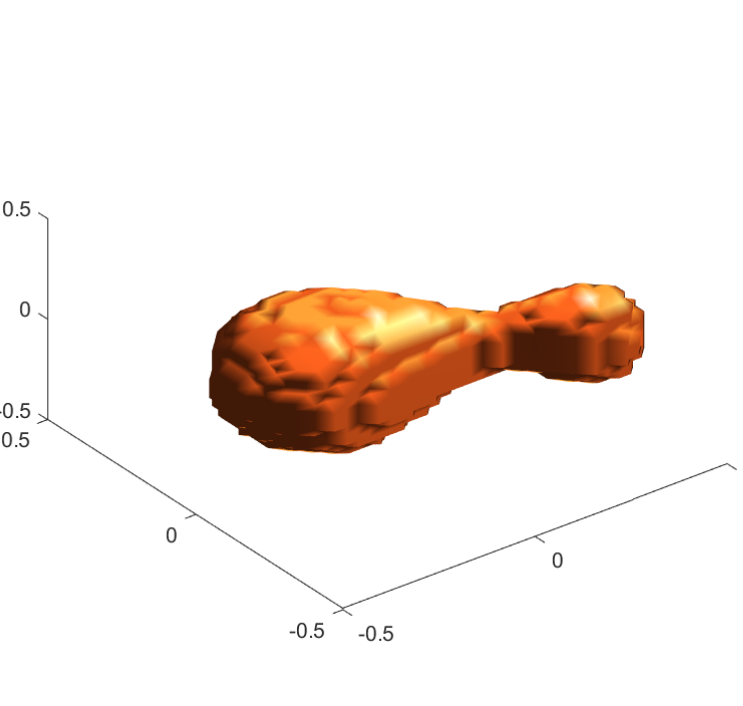}
    \caption{$I_{LSM}$}
    \end{subfigure}
    \quad
    \begin{subfigure}[b]{0.3\textwidth}
    \centering
    \includegraphics[trim={4.5cm 8cm 4.5cm 8cm},clip,scale = 0.3]{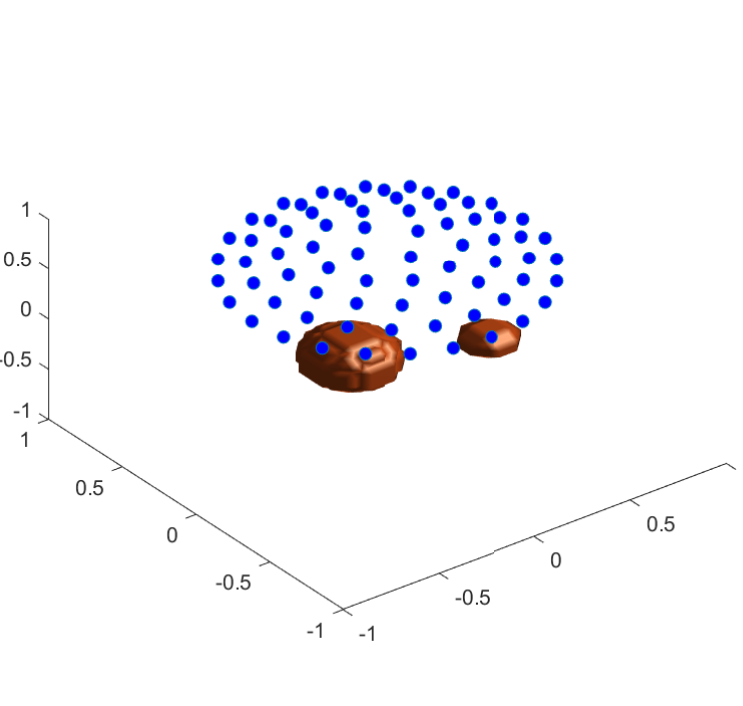}
    \caption{Exact inclusion}
   \end{subfigure}
  \caption{\textbf{Example 6.} Reconstructions by $I_w$ (left) and $I_{LSM}$ (middle) with an aperture of $[0, \frac{\pi}{4}] \times [0, 2\pi]$, $k = 8$, 10\% random noise, and 78 measurement points; and the exact inclusion (right).}
\label{fig_ex3d_1}
\end{figure}

From Fig.\,\ref{fig_ex3d_1}, despite the challenge of reconstructing two inclusions of different sizes that are close to each other, the proposed method (left plot) successfully recovers both the location and shape of the two balls under noisy and limited aperture measurement data.

\textbf{Example 7.} This example, shown in Fig.\,\ref{fig_ex3d_2}, involves reconstructing two rectangular bars located at $[-0.4, -0.2] \times [-0.25, 0.25] \times [-0.05, 0.05]$ and $[0.2, 0.4] \times [-0.25, 0.25] \times [-0.05, 0.05]$.

\begin{figure}[H]
    \centering
     \begin{subfigure}[b]{0.3\textwidth}
    \centering
    \includegraphics[trim={4.5cm 8cm 4.5cm 8cm},clip,scale = 0.3]{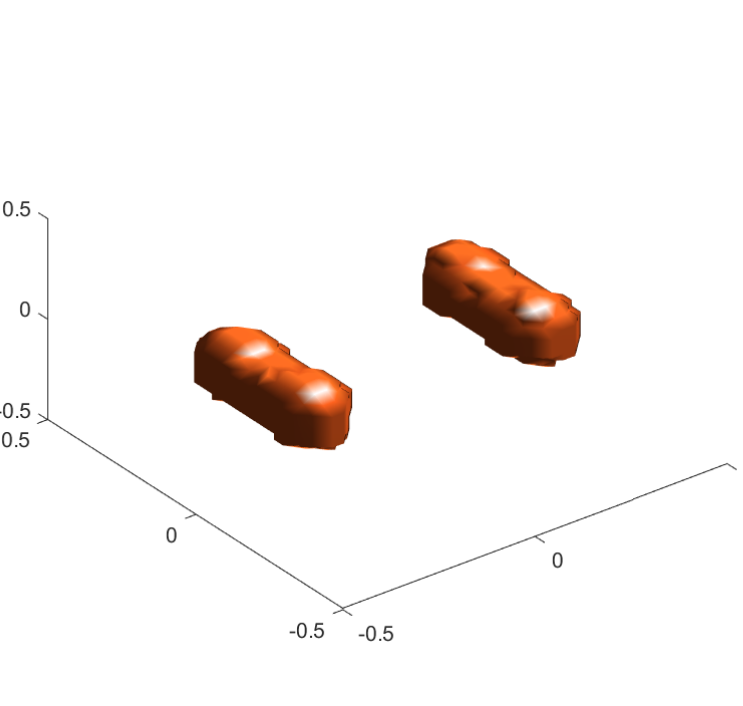}
    \caption{$I_w$}
    \end{subfigure}
    \quad     
    \begin{subfigure}[b]{0.3\textwidth}
    \centering
    \includegraphics[trim={4.5cm 8cm 4.5cm 8cm},clip,scale = 0.3]{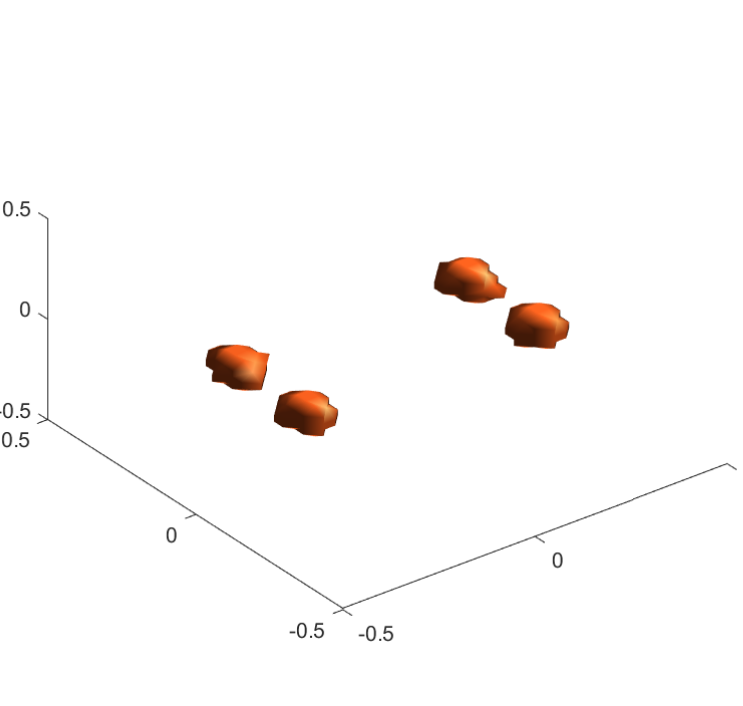}
    \caption{$I_{LSM}$}
    \end{subfigure}
    \quad
    \begin{subfigure}[b]{0.3\textwidth}
    \centering
    \includegraphics[trim={4.5cm 8cm 4.5cm 8cm},clip,scale = 0.3]{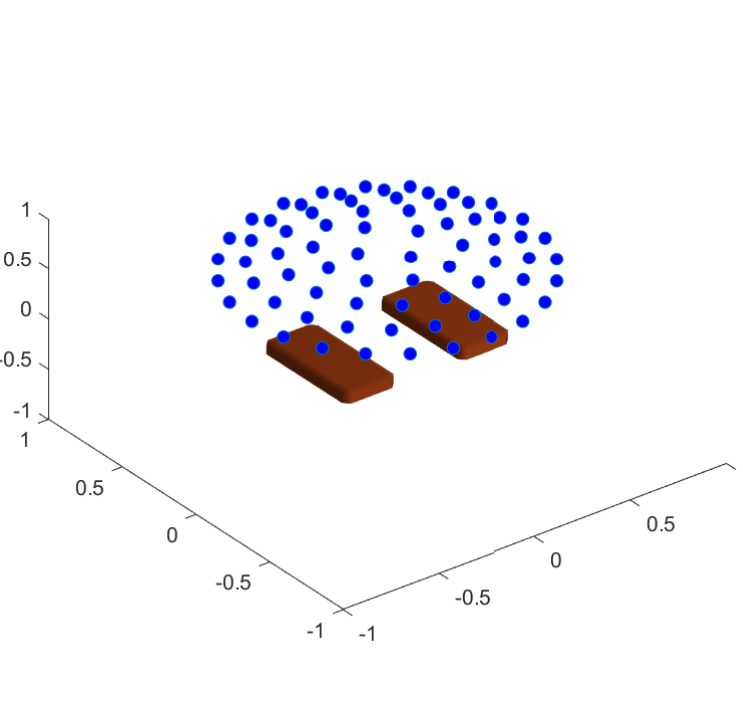}
    \caption{Exact inclusion}
   \end{subfigure}
  \caption{\textbf{Example 7.} Reconstructions by $I_w$ (left) and $I_{LSM}$ (middle) with an aperture of $[0, \frac{\pi}{4}] \times [0, 2\pi]$, $k = 8$, 10\% random noise, and 78 measurement points; and the exact inclusion (right).}
\label{fig_ex3d_2}
\end{figure}

From Fig.\,\ref{fig_ex3d_2}, it is observed that the proposed method (left plot) effectively identifies the shape and recovers the boundaries of the rectangular bars with high robustness.

\section{Conclusions}
In this work, we introduce a novel sampling method to solve the inverse medium scattering problem with acoustic waves, particularly when only limited aperture measurement data is available. This method is inspired by the classical linear sampling method, which performs well with full measurement data, and the direct sampling method, which enhances reconstruction accuracy and stability in limited aperture scenarios. Our approach employs a new weighted inner product between test functions and basis functions for the range of the far-field operator, which simplifies implementation.

To validate the accuracy and stability of the proposed method, we provide a theoretical justification for both cases of infinite and finite measurement data under the Born approximation. Furthermore, extensive numerical experiments demonstrate the superior performance of our method compared to existing sampling-based techniques. We also discuss the method's generalization to near-field imaging and its integration with the factorization method.

Several promising research directions arise from this work. First, it is crucial to rigorously establish the improvements in accuracy and stability offered by the new index function compared to the classical linear sampling method beyond the Born approximation. This could lead to a more general validation of the proposed method.
Second, extending the proposed method to other limited aperture inverse problems, where sampling-type methods have been previously applied, holds significant practical value. This includes full electromagnetic wave scattering problems, electrical impedance tomography, and time-dependent inverse problems.
Third, the core idea of this work involves solving the inverse problem by computing the weighted inner product of test functions and basis functions for specific operators. This concept can be generalized as a framework for solving inverse problems by selecting an appropriate weight function based on the forward operator, which would yield a robust and easily implementable numerical method. Developing and validating a framework for such direct methods in general inverse problems is a compelling future research avenue.

\bibliographystyle{siam}
\bibliography{references_Limited}
\end{document}